\documentclass[11pt,letterpaper]{amsart}
\usepackage{amscd, amsmath, amsthm, amssymb}
\usepackage{color}
\usepackage{graphicx}
\usepackage{mathrsfs}
\usepackage{verbatim}
\usepackage[all]{xy}
\usepackage[pagebackref]{hyperref}
\usepackage{enumitem}
\usepackage{xcolor}

\newtheorem{theorem}{Theorem}[section]
\newtheorem{lemma}[theorem]{Lemma}

\newtheorem{proposition}[theorem]{Proposition}
\newtheorem{conjecture}[theorem]{Conjecture}
\newtheorem{question}[theorem]{Question}
\newtheorem{corollary}[theorem]{Corollary}

\theoremstyle{definition}
\newtheorem{definition}[theorem]{Definition}
\newtheorem{example}[theorem]{Example}

\newtheorem{remark}[theorem]{Remark}

\newtheorem{mainthm}{Theorem}

\numberwithin{equation}{section}

\newcommand\bba{\mathbb{A}} 
\newcommand\CC{{\mathbb{C}}}

\newcommand\PP{{\mathbb{P}}} 
\newcommand\QQ{{\mathbb{Q}}} 
\newcommand\RR{{\mathbb{R}}} 
\newcommand\ZZ{{\mathbb{Z}}}

\newcommand\B{\mathbf{B}} 
\newcommand\kk{\mathbf{k}}

\newcommand\OOO{{\mathscr{O}}}

\newcommand\OO{{\mathcal{O}}}

\newcommand\qb{\overline{\mathbb{Q}}}
\newcommand\Supp{{\rm Supp}\,}

\newcommand\tx{\widetilde{X}}

\DeclareMathOperator{\vol}{vol} 
\DeclareMathOperator{\codim}{codim} 
\DeclareMathOperator{\End}{End} 
\DeclareMathOperator{\Sing}{Sing} 

\begin{document}

	\title[Arithmetic degrees and Zariski dense orbits]
	{Arithmetic degrees and Zariski dense orbits of cohomologically hyperbolic maps} 
	
	\author[Yohsuke Matsuzawa]{Yohsuke Matsuzawa}
	\author[Long Wang]{Long Wang}
	\address{Department of Mathematics, Graduate School of Science, Osaka Metropolitan University, 3-3-138, Sugimoto, Sumiyoshi, Osaka, 558-8585, Japan}
    \email{matsuzaway@omu.ac.jp}

    \address{Shanghai Center for Mathematical Sciences, Fudan University, Jiangwan Campus, Shanghai, 200438, China; Graduate School of Mathematical Sciences, The University of Tokyo, 3-8-1 Komaba, Meguro-Ku, Tokyo 153-8914, Japan}
    \email{wanglll@fudan.edu.cn}

	\begin{abstract} 
		A dominant rational self-map on a projective variety is called $p$-cohomologically hyperbolic if the $p$-th dynamical degree is strictly larger than other dynamical degrees. For such a map defined over $\overline{\mathbb{Q}}$, we study lower bounds of the arithmetic degrees, existence of points with Zariski dense orbit, and finiteness of preperiodic points. In particular, we prove that, if $f$ is an $1$-cohomologically hyperbolic map on a smooth projective variety, then (1) the arithmetic degree of a $\overline{\mathbb{Q}}$-point with generic $f$-orbit is equal to the first dynamical degree of $f$; and (2) there are $\overline{\mathbb{Q}}$-points with generic $f$-orbit. Applying our theorem to the recently constructed rational map with transcendental dynamical degree, we confirm that the arithmetic degree can be transcendental. 
	\end{abstract}

	

	\maketitle
	
	\thispagestyle{empty}


	\section{Introduction}\label{sec_intr}
	
	Let $X$ be a smooth projective variety defined over $\qb$
	and $f \colon X \dashrightarrow X$ a dominant rational map.
	Let us fix an ample divisor $H$ on $X$.
	Then for $k = 0,1,\dots, \dim X$,
	the $k$-th dynamical degree $\delta_{k}(f)$ of $f$ is 
	\begin{align*}
		\delta_{k}(f) := \lim_{n \to \infty} \left( (f^n)^{\ast}(H^k)\cdot H^{\dim X - k} \right)^{1/n}.
	\end{align*}
	See Section \ref{sec_prel} for definitions and basic properties of dynamical degrees.
	These quantities measure geometric complexity of the map $f$.
	In particular, $\delta_{1}(f)$ measures how $(f^n)^*H$ is getting ``larger'' 
	as $n$ goes to infinity.
	The Kawaguchi--Silverman conjecture says that the growth rate of 
	the height function along a Zariski dense orbit would be equal to $\delta_{1}(f)$:
	\begin{align*}
		\alpha_f(x) := \lim_{n \to \infty} \max\{ 1, h_H(f^n(x))\}^{1/n} = \delta_1(f),
	\end{align*}
	where $h_H$ is a logarithmic Weil height function on $X$ associated with $H$.
	(Note that the existence of the limit is a part of the conjecture.)
	This quantity $\alpha_f(x)$ is called the arithmetic degree of $f$ at $x$.
	
	In this paper, we investigate arithmetic degrees of 
	$p$-cohomologically hyperbolic maps. 
	Also, as byproducts, we prove the existence of Zariski dense orbits
	and finiteness of preperiodic points for certain classes of rational maps.
    We say that the $f$-orbit of a point $x \in X(\qb)$ is \emph{well-defined}
    if $f^n(x)$ is not contained in the indeterminacy locus of $f$ for every
    $n \geq 0$, and we say that a well-defined orbit is \emph{generic} if 
    it is infinite and its intersection with any proper Zariski closed subset
    is finite.
	Among other results, we prove that, if $f$ is $1$-cohomologically hyperbolic, then
	\begin{itemize}
		\item $\alpha_f(x) = \delta_1(f)$ for each $\qb$-point $x$ with well-defined and generic $f$-orbit; 
		
		\item there exists a $\qb$-point whose $f$-orbit is well-defined and generic.
	\end{itemize}

	\subsection{Cohomologically hyperbolic maps}
	
	\begin{definition}\label{def_coh}
		Let $X$ be a projective variety over an algebraically closed field of characteristic zero. Let $f : X \dashrightarrow X$ be a dominant rational self-map. 
		We say $f$ is \textit{$p$-cohomologically hyperbolic}, if there is a positive integer $1\leq p \leq \dim X$ such that  
		\begin{equation*}
			\delta_p(f) > \delta_i(f) \ \text{ for all } i \in \{ 1, 2, \dots, \dim X \} \setminus \{p\}. 
		\end{equation*}
		We say $f$ is \textit{cohomologically hyperbolic} if it is \textit{$p$-cohomologically hyperbolic} for some $p$. 
	\end{definition}
	
	Notice that the cohomological hyperbolicity of $f$ implies that $\delta_1(f) > 1$ by (\ref{eq_fdd}) below; in particular, $f$ has infinite order. Moreover, for a birational self-map $f$, it is cohomologically hyperbolic if and only if $\delta_1(f) > 1$ when $\dim X = 2$;
	if and only if $\delta_1(f) \neq \delta_2(f)$ when $\dim X = 3$.

	The notion of cohomological hyperbolicity was originally proposed by Guedj \cite{Gu10} from the viewpoint of complex dynamical systems. This is a cohomological analogy of the hyperbolicity in smooth dynamical systems. Cohomologically hyperbolic self-maps are usually expected to have some particular and pretty dynamical properties, such as the existence of invariant measures and the equidistribution of periodic points. We refer to \cite{Gu10} for more background and discussions. The paper \cite{BHT22} provides a nice overview together with references for known results.

	\subsection{Kawaguchi--Silverman conjecture}
	
	Let $X$ be a projective variety over $\qb$. For a dominant rational self-map $f: X \dashrightarrow X$, we let $I(f) \subset X$ be the indeterminacy locus of $f$.
	\begin{definition}
		We write
		\begin{align*}
			X_f(\qb) = \{ x \in X(\qb) \mid \text{$f^i(x)\notin I(f)$ for all $i \geq 0$} \}.
		\end{align*}
		For each point $x \in X_f(\qb)$, we can define the forward $f$-orbit as
		\begin{align*}
			\OO_f(x) = \{f^i(x) \,:\, i \geq 0\}.
		\end{align*}
		When the ground field is an arbitrary field $\kk$, we make the same definition.
		By a result of Amerik \cite{Am11}, the set $X_f(\qb)$ is Zariski dense in $X$. 
	\end{definition}

	\begin{definition}[cf.\ \cite{KS16a, Si14}]
		Let $H$ be an ample Cartier divisor on $X$ and
		fix a logarithmic Weil height function $h_H \colon X(\qb) \longrightarrow \RR$.
		See, for example, \cite{BG06,HS00,Lan83} for the definition and the basic theory of Weil height functions.
		We set $h_H^+ = \max\{1, h_H\}$.
		For $x \in X_f(\qb)$, we define the \textit{lower} and the \textit{upper arithmetic degree} of $f$ at $x$ by 
		\[ \underline{\alpha}_f(x) = \liminf_{i\to \infty} h^{+}_H(f^i(x))^{1/i}, \]
		\[ \overline{\alpha}_f(x) = \limsup_{i\to \infty} h^{+}_H(f^i(x))^{1/i}. \]
		Both of these quantities are independent of the choice of ample divisor $H$ and height function $h_H$ (\cite[Proposition 12]{KS16a}). 
		When $\underline{\alpha}_f(x) = \overline{\alpha}_f(x)$, we write
		\[ \alpha_f(x) = \lim_{i\to \infty} h^{+}_H(f^i(x))^{1/i}
		\]
		and call it the arithmetic degree.
	\end{definition}

	The following is a conjecture of Kawaguchi and Silverman concerning the properties of arithmetic degrees, which is the main theme of this paper.

	\begin{conjecture}[{\cite[Conjecture 6]{KS16a}, see also \cite[Conjecture 40]{Si14}}]\label{ksc} Let $X$ be a smooth projective variety over $\qb$. Let $f : X \dashrightarrow X$ be a dominant rational self-map and $x \in X_f(\qb)$. Then the following hold.
		
		\medskip $(\mathrm{a})$ $\underline{\alpha}_f(x) = \overline{\alpha}_f(x)$. In other words, the limit 
		\[ \lim_{i\to \infty} h^{+}_H(f^i(x))^{1/i} \]
		defining $\alpha_f(x)$ exists.
		
		\medskip $(\mathrm{b})$ $\alpha_f(x)$ is an algebraic integer.
		
		\medskip $(\mathrm{c})$ The collection of arithmetic degrees
		\[ \{ \alpha_f(y) \,:\, y\in X_f(\qb) \} \]
		is a finite set.
		
		\medskip $(\mathrm{d})$ If the forward orbit $\OO_f(x)$ is Zariski dense in $X$, then
		\[ \alpha_f(x) = \delta_1(f). 
		\] 
	\end{conjecture}
	
	In general, we have 
	\[ 1 \leq \underline{\alpha}_f(x) \leq \overline{\alpha}_f(x) \leq \delta_1(f), 
	\]
	where the first two inequalities are obvious, while the last one is non-trivial and is due to \cite[Theorem 1.4]{Ma20} and \cite[Proposition 3.11]{JSXZ21}. 
	In particular, if $\delta_1(f) = 1$, then $\alpha_f(x) = 1 = \delta_1(f)$ for all $x \in X_f(\qb)$. 
	Therefore, we are only interested in the case where $\delta_1(f) > 1$. Moreover, in order to show the existence of $\alpha_f(x)$ and the equality $\alpha_f(x) = \delta_1(f)$, it is enough to show the inequality $\underline{\alpha}_f(x) \geq \delta_1(f)$.

	When $f: X \to X$ is a surjective endomorphism, the first three claims (a), (b) and (c) were shown in \cite{KS16b}, while the last one (d) is still widely open. We refer to \cite{CLO22, MSS18, MMSZ20} and references therein for known results.

	When $f: X \dashrightarrow X$ is an arbitrary dominant rational self-map, much less is known. In fact, 
	there is a counterexample to Conjecture \ref{ksc}(c) due to \cite{LS20}.  
	Also as it turned out that the dynamical degree can be a transcendental number \cite{BDJ20}, 
	part (b) must fail for points with Zariski dense orbit for such maps if part (d) is correct. 
	Indeed, Theorem \ref{thm_trans_two} below gives the first counterexample to Conjecture \ref{ksc}(b). We refer to \cite{Wa22} and references therein for previous works towards claim (d), which is still open for birational self-maps of $\PP^2$. The main obstacle is the deficiency of functoriality of height functions (\cite[Theorem B.3.2.(b)]{HS00}) for an arbitrary rational map.

	\subsection{Main theorems}
	Before giving the statements of our main results, let us introduce two more notions. 
	
	\begin{definition}[{cf.\ \cite[Definition 2.4]{Xi22}}] 
    Let $X$ be a projective variety over an algebraically closed field 
    $\mathbf{k}$ of characteristic zero. 
    Let $f : X \dashrightarrow X$ be a dominant rational self-map, 
    and let $x \in X_{f}(\kk)$. 
    We say $(X, f, x)$ has the \textit{DML property} (short for ``dynamical Mordell--Lang''),
    if for every closed subvariety $Z \subset X$, the set 
		\begin{align}\label{dml-return-set}
			\left\{ n \in \ZZ_{\geq 0} : f^n(x) \in Z \right\} 
		\end{align}
		is a finite union of one-sided arithmetic progressions.
        Here an one-sided arithmetic progression is a set of the form
        $\{ ak + b \mid k \in \ZZ_{\geq 0} \}$ for some non-negative integers $a$ and $b$.
        Note that we allow $a=0$, that is, a set consisting of a single integer is allowed.
	\end{definition}
	
	The \textit{dynamical Mordell--Lang conjecture} (cf. \cite[Conjecture 1.2]{BGT10}) predicts that $(X, f, x)$ has the DML property for every $x \in X_{f}(\kk)$. Although this conjecture is open in general, it is known for \'etale dominant morphisms of quasi-projective varieties (\cite[Theorem 1.3]{BGT10}).

	\begin{definition}\label{def_generic} Let $X$ be a projective variety. A subset $\OO \subset X$ is called \textit{generic}, if $\OO$ is infinite and $\OO\cap Z$ is finite for every proper Zariski closed subset $Z \subsetneq X$.
	\end{definition}

	Clearly, a generic subset $\OO \subset X$ is Zariski dense in $X$. Conversely, the dynamical Mordell--Lang conjecture implies that a Zariski dense orbit $\OO_f(x)$ is generic in $X$. 
	
	The following is our main theorem, which generalizes results from \cite{Wa22}. 
	
	\begin{mainthm}\label{thm_lb_pch} Let $f: X \dashrightarrow X$ be a dominant rational self-map of a smooth projective variety $X$ defined over $\qb$. Assume that $f$ is $p$-cohomologically hyperbolic. 
		
		\smallskip $(1)$ For every $x \in X_f(\qb)$ with generic orbit $\OO_f(x)$, we have 
		\[ \underline{\alpha}_f(x) \geq \frac{\delta_p(f)}{\delta_{p-1}(f)}. 
		\]
		
		\smallskip $(2)$ There exists a sequence $\{x_i\}_{i\geq 1} \subset X_f(\qb)$ of $\qb$-points, such that $(X, f, x_i)$ has the DML property for each $i$, and the sequence $\{\underline{\alpha}_f(x_i)\}_{i\geq 1}$ of lower arithmetic degrees converges and satisfies 
		\[ \lim_{i\to \infty} \underline{\alpha}_f(x_i) \geq \frac{\delta_p(f)}{\delta_{p-1}(f)}.
		\]
	\end{mainthm}
	
	\begin{remark}\label{rmk:existence-points-adelic-open-ver}
        In Theorem \ref{thm_lb_pch} (2), we do not claim that the orbits of $x_i$'s are disjoint.
		What we actually prove is that for an arbitrary $\epsilon>0$, there is a non-empty adelic open subset $A \subset X(\qb)$ such that
		\begin{align*}
			\underline{ \alpha}_f(x) \geq \frac{\delta_p(f)}{\delta_{p-1}(f)} -\epsilon
		\end{align*}
		for all $x \in A$.
		Here ``adelic open subset'' is in the sense of \cite[Section 3]{Xi22}.
        The formal definition of adelic topology involves quite a lot.
        Roughly speaking, for each embedding of $\qb$ into
        $\CC_p$ for any $p \leq \infty$, we consider the inverse image $U$ of
        $p$-adic open subsets to $X(\qb)$. 
        Then take the union $V$ of all Galois conjugates of $U$.
        These $V$ as well as the images of such sets by flat morphisms from 
        other varieties are the basis of the adelic topology.
        (To make sense this definition, we need to care about the
        different choices of models of $X$ over a number field.)
        It is worth to note that two non-empty adelic open sets of an irreducible variety have a non-empty intersection. Therefore, the adelic topology allows us to combine results obtained by
        $p$-adic analytic methods for finitely many different $p$.
        Also we can use the terminology ``general" for adelic opens
        in the same way as for Zariski opens.
	\end{remark}
	
	In the proof of Theorem \ref{thm_lb_pch}, we use Xie's result 
	(\cite[Propositions 3.24 and 3.27]{Xi22}) to find points with DML property.
	This is one of new ingredients we use after \cite{Wa22}.
	As a direct corollary, we have

	\begin{corollary}\label{thm_lb_onech} 
        Let $f: X \dashrightarrow X$ be a dominant rational self-map of a smooth projective variety $X$ defined over $\qb$. Assume that $f$ is $1$-cohomologically hyperbolic. 
		\smallskip $(1)$ For every $x \in X_f(\qb)$ with generic orbit $\OO_f(x)$,  $\alpha_f(x)$ exists and $\alpha_f(x) = \delta_1(f)$. 
		\smallskip $(2)$ There exists a sequence $\{x_i\}_{i\geq 1} \subset X_f(\qb)$ of $\qb$-points, such that $(X, f, x_i)$ has the DML property for each $i$, and the sequence $\{\underline{\alpha}_f(x_i)\}_{i\geq 1}$ of lower arithmetic degrees converges to $\delta_1(f)$: 
		\[ \lim_{i\to \infty} \underline{\alpha}_f(x_i) = \delta_1(f). 
		\]
	\end{corollary}

	As an application of Corollary \ref{thm_lb_onech}, we study the following so-called Zariski dense orbit conjecture for dominant rational self-maps. The argument here is inspired by \cite[Section 8]{JSXZ21}. 
	
	\begin{conjecture}\label{conj_zdo} Let $X$ be a projective variety over an algebraically closed field $\mathbf{k}$ of characteristic zero, and let $f : X \dashrightarrow X$ be a dominant rational self-map. If every $f$-invariant rational function on $X$ is constant, then there exists $x \in X_f(\kk)$ whose orbit $\OO_f(x)$ is Zariski dense in $X$.
	\end{conjecture}
	
	We refer to \cite[Introduction]{Xi22} for the history of this conjecture and the state of the art. We note that a cohomologically hyperbolic dominant rational self-map does not admit invariant rational functions by the product formula of dynamical degrees (see Lemma \ref{lem_no_invariant_fib}). 
	So cohomological hyperbolicity should be enough for the existence of Zariski dense orbit.
	For $1$-cohomologically hyperbolic maps, we can confirm it:
	
	\begin{mainthm}\label{thm_zdo}
		Let $X$ be a smooth projective variety defined over $\qb$.
		Let $f \colon X \dashrightarrow X$ be a $1$-cohomologically hyperbolic
		dominant rational self-map.
		Then there is $x \in X_f(\qb)$ such that the $f$-orbit $\OO_f(x)$
		is Zariski dense in $X$. Moreover, we can take such $x$ so that $\alpha_f(x)$ exists and $\alpha_f(x) = \delta_1(f)$. 
	\end{mainthm}
	\begin{remark} The first claim is a generalization of \cite[Theorem 1.12(3)]{JSXZ21} where $f$ is assumed to be a morphism. See Proposition \ref{prop:exist-zdo-general-case} for a more general statement. 
	\end{remark}

	The following is a partial generalization of \cite[Theorem 1.12(2)]{JSXZ21} where $f$ is an automorphism of a projective threefold with first dynamical degree greater than one. 
	
	\begin{mainthm}\label{thm_zdo_three} Let $f: X \dashrightarrow X$ be a cohomologically hyperbolic birational self-map of a smooth projective threefold $X$ defined over $\qb$. 
		Then there is a $\qb$-point $x \in X_f(\qb)$ such that the $f$-orbit $\OO_f(x)$ 
		is Zariski dense in $X$. 
	\end{mainthm}

	Combining Theorem \ref{thm_zdo} with the striking result of Bell--Diller--Jonsson \cite{BDJ20}, we see that arithmetic degrees can be transcendental for dominant rational self-maps in dimension two. This gives the first counterexample to Conjecture \ref{ksc}(b). 
	
	\begin{mainthm}\label{thm_trans_two} There is a dominant rational self-map $f: \PP^2 \dashrightarrow \PP^2$ defined over $\qb$ and a $\qb$-point $x\in \PP^2_f(\qb)$ such that $\alpha_f(x)$ exists and $\alpha_f(x) = \delta_1(f)$ is a transcendental number. 
	\end{mainthm}

	If the answer of the following question is affirmative, then for every $d \geq 3$, there is a birational self-map $f: \PP^d \dashrightarrow \PP^d$ defined over $\qb$ and a $\qb$-point $x\in \PP^d_f(\qb)$ such that $\alpha_f(x)$ exists and $\alpha_f(x) = \delta_1(f)$ is a transcendental number. 
	
	\begin{question}[{cf. \cite[Question 1.6 and Theorem 1.1]{BDJK21}}]\label{ques_trans} Let $d \geq 3$ be a positive integer. Does there exist a birational map $f : \PP^d \dashrightarrow \PP^d$ defined over $\qb$ which is $1$-cohomologically hyperbolic and for which $\delta_1(f)$ is transcendental? 
	\end{question}

	\begin{remark} We notice that for a birational self-map of $\PP^2$, the first dynamical degree is always an algebraic integer (\cite{DF01}). If (a) and (d) of Conjecture \ref{ksc} are true for a birational self-map of $\PP^2$, then it is not difficult to see that its arithmetic degrees are all algebraic integers. 
	\end{remark}

	Finally we give two theorems for dominant rational self-maps with ``large topological degree'', namely, $(\dim X)$-cohomologically hyperbolic self-maps. Such maps are ``highly non-invertible'' and it allows us to prove the following type of inequality:
	\begin{align*}
		h_H^+(f(x))  \geq \kappa h_H^+(x) - C \quad \text{outside a closed subset $Z$}
	\end{align*}
	where $\kappa$ would be some positive real or $1$ depending on the size of $Z$. Combining this idea with a weak dynamical Mordell--Lang theorem \cite{BGT15, BHS20} (cf. \cite[Proof of Proposition 3.11]{JSXZ21}), we can prove:

	\begin{mainthm}\label{thm:ad-largetopdeg}
		Let $X$ be a smooth projective variety of dimension $N$ defined over $ \overline{\QQ}$, and let $f \colon X \dashrightarrow X$ be an $N$-cohomologically hyperbolic dominant rational self-map. Let $x \in X_{f}( \overline{\QQ})$ with Zariski dense $f$-orbit $\OO_{f}(x)$ in $X$. Assume that 
		\begin{itemize}\label{assump_intersection}
			\item[$(\ast)$] the intersection of $\OO_{f}(x)$ and any Zariski closed subset of $X$ of codimension at least two is finite.
		\end{itemize}
		Then we have
		\begin{align*}
			\underline{\alpha}_{f}(x) \geq \frac{ \delta_{N}(f)}{ \delta_{N-1}(f)}.
		\end{align*}
	\end{mainthm}

	We note that the assumption $(\ast)$ is automatically satisfied when $\dim X = 2$.

	It is known that the set of preperiodic points of a non-invertible surjective self-morphism
	on a projective space is a set of bounded height, i.e., any height function associated with
	an ample divisor is bounded on this set.
	The second author proved such finiteness of preperiodic points for birational cohomologically
	hyperbolic maps \cite[Theorem 1.1]{Wa22}.
	Following his idea, we prove an $N$-cohomologically hyperbolic version:

	\begin{mainthm}\label{thm:fin-preper-largetopdeg}
		Let $X$ be a smooth projective variety of dimension $N$ defined over $\qb$.
		Let $f \colon X \dashrightarrow X$ be an $N$-cohomologically hyperbolic dominant rational map.
		Then there is a non-empty Zariski open subset $U \subset X$ such that
		the set
		\begin{align*}
			\left\{ x \in U(\qb) \cap X_f(\qb) \,\middle|\, \OO_f(x) \subset U, \# \OO_f(x) < \infty \right\}
		\end{align*}
		is a set of bounded height.
	\end{mainthm}

	\subsection*{Organization of the paper}
	In Section \ref{sec_prel}, we summarize properties of dynamical degree and arithmetic degree.
	Some of them are well-known facts and some of them are new (at least we could not find references).
	In Section \ref{sec_lemmas}, we show several key lemmas for the proof of main theorems.
	In Section \ref{sec_prf}, we prove main theorems except the last two ones about self-maps with large topological degree. We also collect several examples of cohomologically hyperbolic self-maps. The last two theorems, i.e., Theorems \ref{thm:ad-largetopdeg} and \ref{thm:fin-preper-largetopdeg}, are proved in Section \ref{sec_lgtopdeg}.

	\subsection*{Convention}  
	\begin{itemize}
		\item A variety is an irreducible and reduced separated scheme of finite type over a field. 
		\item Let $X$ be a projective variety over $\qb$ and $D$ an $\RR$-Cartier divisor on $X$.
		By the basic theory of Weil height functions, we can attach $D$ a function $X(\qb) \to \RR$
		called a Weil height function, which is determined by $D$ up to difference by bounded functions. By fixing a Weil height function $h_D$ associated with $D$, we mean picking and fixing one such function. 
		\item Let $X$ be a reduced separated scheme of finite type over $\qb$.
		{\it The adelic topology on $X(\qb)$} is the topology on $X(\qb)$ defined in \cite[Section 3]{Xi22}.
		Adelic topology is stronger than Zariski topology, and irreducible components with respect to
		Zariski topology are still irreducible components with respect to adelic topology.
		See \cite[Proposition 3.18]{Xi22} for more properties of this topology.
	\end{itemize}
	
	\subsection*{Acknowledgements.} 
	This project started when the authors stayed at National University of
	Singapore hosted by Professor De-Qi Zhang. We wish to thank him for his hospitality.
	The first author thanks Doctor Shou Yoshikawa for answering his questions. The second author thanks Professors Chen Jiang and Keiji Oguiso as well as Doctors Jia Jia and Sichen Li for suggestions, discussions and comments. Both authors thank Professors Tien-Cuong Dinh,  Joseph Silverman, Junyi Xie, and the referee for their comments and suggestions. 
	The first author is supported by JSPS KAKENHI Grant Number JP22K13903.
	The second author is supported by JSPS KAKENHI Grant (21J10242), Postdoctoral Fellowship Program of CPSF (GZC20230535), and National Key Research and Development Program of China (\#2023YFA1010600).

	\section{Preliminaries}\label{sec_prel}

	\subsection{Dynamical degrees} 
	In this subsection, we work over an algebraically closed field $\kk$
	of characteristic zero.
	Let $X$ be a smooth projective variety over $\kk$. Let $f : X \dashrightarrow X$ be a dominant rational self-map. 
	\begin{definition}\ 
		\begin{enumerate}
			\item 
			Let us fix a resolution of indeterminacy
			\begin{align*}
				\xymatrix{
					& X' \ar[ld]_{\psi} \ar[rd]^{\phi} &\\
					X \ar@{-->}[rr]_f && X
				}
			\end{align*}
			i.e. a smooth projective variety $X'$ with a birational morphism $\psi$
			and a morphism $\phi$ such that $f \circ \psi = \phi$.
			For a codimension $k$-cycle $\alpha$ on $X$, we define 
			the pull-back cycle class $f^* \alpha$ as
			\begin{align*}
				f^* \alpha = \psi_* \phi^* \alpha
			\end{align*}
			where $\phi^*$ and $\psi_*$ in the right-hand side are
			usual Gysin pull-back and proper push-forward of cycles.
			Note that, by the projection formula,  $f^* \alpha$ is independent of the choice of the resolution $X'$. 
			This $f^*$ induces a homomorphism
			\begin{align*}
				f^* \colon N^k(X) \longrightarrow N^k(X)
			\end{align*}
			where $N^k(X)$ is the group of codimension $k$-cycles on $X$ modulo
			numerical equivalence.
			
			\item 
			Let $H$ be an ample Cartier divisor on $X$.
			For $k = 0,1, \dots, \dim X$, the $k$-th degree of $f$ with respect to $H$ is
			\begin{align*}
				\deg_{k,H}(f) = \left( f^{\ast}(H^k)\cdot H^{\dim X - k} \right).
			\end{align*}
			The \textit{$k$-th dynamical degree} of $f$ is defined as
			\begin{align*}
				\delta_k(f) = \lim_{i \to \infty} \left( \deg_{k, H} (f^i) \right)^{1/i}.
			\end{align*}
			This limit exists and is independent of the choice of ample Cartier divisor $H$ (\cite{Da20, Tr20}). In fact, we have
			\begin{align*}
				\delta_{k}(f) = \lim_{i \to \infty} \|(f^i)^*|_{N^k(X)_\RR} \|^{1/i}
			\end{align*}
			where $N^k(X)_\RR = N^k(X) \otimes_\ZZ \RR$, $(f^i)^*|_{N^k(X)_\RR}$
			is the pull-back linear map acting on $N^k(X)_\RR$, and $\|\ \|$
			is an arbitrary norm on the finite dimensional vector space 
			$\End_\RR(N^k(X)_\RR)$ (cf. \cite{Da20,Tr20}).
		\end{enumerate}
		
	\end{definition}
	
	In particular, when $f$ is a surjective endomorphism, or more generally, $f$ is \textit{algebraically} $k$-\textit{stable}, that is, $(f^i)^{\ast} = (f^{\ast})^i$ on $N^k(X)_{\RR}$ for all $i \geq 1$, we have
	\[ \delta_k(f) = \text{spectral radius of}\ f^{\ast}: N^k(X)_{\RR} \to N^k(X)_{\RR}. \]
	
	One basic property of dynamical degrees is the birational invariance. Precisely, if $\pi: X \dashrightarrow Y$ is a generically finite dominant rational map between normal projective varieties, and $g : Y \dashrightarrow Y$ is a rational self-map such that $g \circ \pi = \pi \circ f$, then
	$\delta_i(f) = \delta_i(g)$ for all $i$ (see e.g., \cite[Theorem 1]{Da20}). In view of this, we can define the dynamical degrees for dominant rational self-maps of an \textit{arbitrary} projective variety in an obvious manner. 
	
	Another basic property is that the function $k \mapsto \log \delta_k(f)$ is concave (see e.g., \cite[Proposition 1.2.i)]{Gu05}). Equivalently,  we have 
	\begin{equation}\label{eq_logc}
		\delta_k(f)^2 \geq \delta_{k-1}(f) \delta_{k+1}(f) \ \ \text{for } 1 \leq k \leq \dim X - 1.
	\end{equation}
	In particular, we have 
	\begin{equation}\label{eq_fdd}
		1 \leq \delta_k(f) \leq \delta_1(f)^k \ \ \text{for } 0 \leq k \leq \dim X,  
	\end{equation}
	and there are positive integers $p$ and $q$ with $0 \leq p \leq q \leq \dim X$ such that
	\begin{equation}\label{eq_kt}
		1 = \delta_0(f) < \cdots < \delta_p(f) = \cdots = \delta_q(f) > \cdots > \delta_{\dim X}(f). 
	\end{equation}

	The following is a consequence of log concavity of dynamical degrees.
	We use it in the proof of Proposition \ref{prop_bigness}. 
	
	\begin{lemma}\label{lemma:ineq-of-dynamical-degrees} Let $X$ be a projective variety of dimension $N$ over $\kk$ and $f \colon X \dashrightarrow X$ a dominant rational self-map. 
		Let $p \in \{1, \dots, N\}$ be a positive integer such that 
		$\delta_{p}(f) = \max\{\delta_{1}(f),\dots, \delta_{N}(f)\}$.
		Then we have
		\begin{align*}
			\delta_{p}(f) \delta_{p-1}^{p}(f) \delta_{p+1}^{N-p}(f) \geq \delta_{p-1}^{i}(f) \delta_{p+1}^{N-1-i}(f) \delta_{i}(f) \delta_{i+1}(f)
		\end{align*}
		for $i = 0,\dots, N-1$. Here we set $ \delta_{N+1}(f) = 1$.
	\end{lemma}

	\begin{proof}
		Let us write 
		\begin{align*}
			\delta_{i} = \delta_{i}(f)
		\end{align*}
		for $i = 0,\dots , N$. 
		By log concavity of dynamical degrees 
		(\ref{eq_logc}), 
		we have
		\begin{align*}
			\frac{ \delta_{p}}{ \delta_{i}} \geq 
			\begin{cases}
				\left( \frac{ \delta_{p}}{ \delta_{p-1}}\right)^{p-i} \quad \text{if $i < p$}\\[2mm]
				\left( \frac{ \delta_{p}}{ \delta_{p+1}}\right)^{i-p} \quad \text{if $i \geq p$}.
			\end{cases}
		\end{align*}
		Let 
		\begin{align*}
			R_{i} &:= \frac{\delta_{p} \delta_{p-1}^{p} \delta_{p+1}^{N-p} }{ \delta_{p-1}^{i} \delta_{p+1}^{N-1-i} \delta_{i}\delta_{i+1}}\\
			&= \frac{ \delta_{p} \delta_{p+1}}{ \delta_{i} \delta_{i+1}} \left(\frac{ \delta_{p-1}}{ \delta_{p+1}} \right)^{p-i}.
		\end{align*}
		When $i < p$, we have
		\begin{align*}
			R_{i} & = \frac{\delta_{p}}{\delta_{i} }\frac{  \delta_{p+1}}{ \delta_{i+1}} \left(\frac{ \delta_{p-1}}{ \delta_{p+1}} \right)^{p-i} \\
			&\geq \left( \frac{ \delta_{p}}{ \delta_{p-1}}\right)^{p-i} \frac{  \delta_{p+1}}{  \delta_{i+1}} \left(\frac{ \delta_{p-1}}{ \delta_{p+1}} \right)^{p-i}\\
			&= \left( \frac{ \delta_{p}}{ \delta_{p+1}}\right)^{p-i} \frac{  \delta_{p+1}}{  \delta_{i+1}} \geq \frac{ \delta_{p}}{ \delta_{i+1}} \geq 1.
		\end{align*}
		When $i \geq p$, we have
		\begin{align*}
			R_{i} & = \frac{\delta_{p}}{\delta_{i+1} }\frac{  \delta_{p+1}}{ \delta_{i}} \left(\frac{ \delta_{p-1}}{ \delta_{p+1}} \right)^{p-i} \\
			&\geq \left( \frac{ \delta_{p}}{ \delta_{p+1}}\right)^{i+1-p}  \frac{  \delta_{p+1}}{ \delta_{i} } \left(\frac{ \delta_{p-1}}{ \delta_{p+1}} \right)^{p-i}\\
			& = \frac{ \delta_{p}}{ \delta_{i}} \left( \frac{ \delta_{p}}{ \delta_{p-1}}\right)^{i-p} \geq 1.
		\end{align*}
		This proves the lemma. 
	\end{proof}
	
	The following lemma compares dynamical degrees of maps and those of maps restricted on
	an invariant subvariety.
	We use this in the proof of Theorem \ref{thm_zdo}.
	
	\begin{lemma}\label{lem:comp-dyn-deg} 
		Let $X$ be a smooth projective variety of dimension $N$ defined over $\kk$, and $f : X \dashrightarrow X$ a dominant rational self-map. 
		Let $Z \subset X$ be a positive dimensional irreducible closed subvariety of dimension $N-l$ with generic point $\eta$.
		Suppose $\eta \notin I(f)$, $f(\eta)= \eta$ (i.e., $Z \not\subset I(f)$ and $f|_Z$ induces a dominant rational self-map of $Z$), and $f$ is quasi-finite at $\eta$.
		Let $g \colon Z \dashrightarrow Z$ be the rational map induced by $f$.
		Then we have 
		\begin{align*}
			\delta_{i}(g) \leq \delta_{i+l}(f)
		\end{align*}
		for $i=1,\dots, N-l$.
	\end{lemma}
	\begin{proof}
		Take a very ample divisor $H$ on $X$ with the following property:
		there are effective divisors $H_1, \dots ,H_l$ that are linearly equivalent to $H$
		such that $H_1 \cap \cdots \cap H_l$ contains $Z$ as an irreducible component.
		
		Fix arbitrary $n \in \ZZ_{\geq 1}$.
		Take the following commutative diagram:
		\begin{align*}
			\xymatrix{
				Z \ar@{}[d]|{\bigcap} \ar@{-->}@/^20pt/[rr]^{g^n} & \widetilde{Z}  \ar[l]_{\psi_Z} \ar@{}[d]|{\bigcap} \ar[r]^{\phi_Z}& Z \ar@{}[d]|{\bigcap} \\
				X \ar@{-->}@/_20pt/[rr]_{f^n} & \widetilde{X} \ar[l]^{\psi} \ar[r]_\phi & X
			}
		\end{align*}
		where $\widetilde{X}$ is a smooth projective variety, $\psi$ is isomorphic over 
		$X \setminus I(f^n)$, $\widetilde{Z}$ is an irreducible closed subvariety of $\widetilde{X}$,
		and $\psi_Z$ is isomorphic over $Z \setminus I(f^n)$.
		Let $H_Z$ be the restriction of $H$ on $Z$ (as a divisor class).
		Then we have
		\begin{align*}
			\deg_{i,H_Z}(g^n) &= \left((g^n)^*(H_Z^i)\cdot H_Z^{N-l-i} \right)_{\widetilde{Z}} \\
			&= \left(\phi_Z^*H_Z^i \cdot \psi_Z^* H_Z^{N-l-i} \right)_{\widetilde{Z}}\\
			&= \left( \phi^*H^i \cdot \psi^*H^{N-l-i} \cdot [\widetilde{Z}]\right)_{\widetilde{X}}.
		\end{align*}
		By our assumptions, $\eta \notin I(f^n)$ and $f^n$ is quasi-finite at $\eta$.
		Thus $\phi$ is quasi-finite at the generic point of $\widetilde{Z}$.
		This implies $\phi^*H_1 \cap \cdots \cap \phi^*H_l$ contains $\widetilde{Z}$
		as an irreducible component.
		By \cite[Lemma 3.3]{JSXZ21}, $\phi^*H_1 \cdots \phi^*H_l - [\widetilde{Z}]$
		defines a pseudo-effective class in $N^l(X)$, the group of codimension $l$-cycles
		modulo numerical equivalence.
		Thus we obtain
		\begin{align*}
			\left( \phi^*H^i \cdot \psi^*H^{N-l-i} \cdot [\widetilde{Z}]\right)_{\widetilde{X}}
			&\leq \left( \phi^*H^i \cdot \psi^*H^{N-l-i} \cdot \phi^*H_1 \cdots \phi^*H_l \right)_{\widetilde{X}} \\
			&= \left( \phi^*H^{i+l} \cdot \psi^*H^{N-l-i}  \right)_{\widetilde{X}} \\
			&= \deg_{i+l,H}(f^n).
		\end{align*}
		It then follows that $\deg_{i,H_Z}(g^n) \leq \deg_{i+l,H}(f^n)$ for all $n\geq 1$
		and hence $\delta_{i}(g) \leq \delta_{i+l}(f)$ as claimed.
	\end{proof}

    In Lemma \ref{lem:comp-dyn-deg}, the assumption that $f$ is quasi-finite at $\eta$ roughly means that $Z$ is not contained in the locus where $f$ is not generically finite. We need this assumption for our proof, and do not know if we can get rid of this. As a contrast, let us recall the following similar result, a special case of \cite[Proposition 3.2]{JSXZ21}, which also compares the dynamical degrees of a rational self-map and of its restriction.

	\begin{lemma}\label{lem_jsxz3.2} Let $X$ be a smooth projective variety over $\kk$, and $f : X \dashrightarrow X$ a dominant rational self-map. Let $Z$ be an irreducible closed subvariety of $X$ which is not contained in $I(f)$ such that $f|_Z$ induces a dominant rational self-map of $Z$. Then for $i = 0, 1, \dots, \dim Z$, we have 
		\[ \delta_i(f|_Z) \leq \delta_i(f). 
		\]
	\end{lemma}

	We end this subsection with the following well-known fact saying that a cohomologically hyperbolic dominant rational self-map does not admit any invariant rational function (cf. \cite[Theorem 4.1]{OT15}). 
	
	\begin{lemma}\label{lem_no_invariant_fib} Let $X$ be a projective variety over $\kk$, and $f : X \dashrightarrow X$ a cohomologically hyperbolic dominant rational self-map. Then there exists no $f$-invariant rational function on $X$. 
	\end{lemma}
	
	\begin{proof} Suppose that there exists an $f$-invariant rational function on $X$ given by the following commutative diagram of dominant rational maps
		\begin{align*}
			\xymatrix{
				X \ar@{-->}[rr]^f \ar@{-->}[rd]_{\pi} & & X.   \ar@{-->}[ld]^{\pi}\\
				&\PP^1&
			}
		\end{align*}
		We may assume that $X$ is smooth. By the product formula of dynamical degrees, 
		\begin{align*}
			\delta_{i}(f) = 
			\begin{cases}
				\max\{ \delta_{i}(f|\pi), \delta_{i-1}(f|\pi)\},  &\text{if $1 \leq i \leq \dim X -1$}\\
				\delta_{\dim X -1}(f|\pi), &\text{if $i = \dim X$,}
			\end{cases}
		\end{align*}
		where $\delta_{i}(f|\pi)$ is the $i$-th relative dynamical degree (see \cite[Section 5]{Da20}).
		By this expression, we see that $f$ cannot be cohomologically hyperbolic. This is a contradiction. 
	\end{proof}

	\subsection{Arithmetic degrees} 
	In this section, we work over $\qb$.
	Let $X$ be a projective variety over $\qb$, and $f: X \dashrightarrow X$
	a dominant rational self-map.
	We collect some basic facts here about arithmetic degree for later use. The following inequality is fundamental. 
	
	\begin{proposition}\label{prop_jsxz3.11} 
		Let $X$ be a projective variety over $\qb$. Let $f : X \dashrightarrow X$ be a dominant rational self-map and $x \in X_f(\qb)$. Then $\overline{\alpha}_f(x) \leq \delta_1(f)$. 
	\end{proposition}
	
	\begin{proof} This is proved for smooth case in \cite[Theorem 1.4]{Ma20}, and for general case in \cite[Proposition 3.11]{JSXZ21}. 
	\end{proof}

	\begin{lemma}\label{lem_ks13} Let $f : X \dashrightarrow X$ be a dominant rational self-map of a projective variety $X$ defined over $\qb$. Let $x\in X_f(\qb)$ and $k$, $n$ non-negative integers. 
		
		\begin{itemize}
			\item[$(1)$] $\underline{\alpha}_f(f^k(x)) = \underline{\alpha}_f(x)$ and  $\overline{\alpha}_f(f^k(x)) = \overline{\alpha}_f(x)$. 
			
			\item[$(2)$] $\underline{\alpha}_{f^n}(x) = \underline{\alpha}_{f}(x)^n$ and $\overline{\alpha}_{f^n}(x) = \overline{\alpha}_{f}(x)^n$.
		\end{itemize}
	\end{lemma}
	
	\begin{proof} (1) This is clear by definition (cf. \cite[Lemma 13]{KS16a}). 
		
		(2) This is well-known (cf. \cite[Lemma 2.4]{Wa22} and references therein).
	\end{proof}

	\begin{lemma}[{cf. \cite[Lemma 2.5]{MMSZ20}}]\label{lem_mmsz2.5} Let $f : X \dashrightarrow X$ be dominant rational self-map on a projective variety $X$ defined over $\qb$ and $W$ an irreducible closed subvariety of $X$ such that $f|_W$ induces a dominant rational self-map of $W$.  
		Then 
		\[ \overline{\alpha}_{f|_W}(x) = \overline{\alpha}_{f}(x) \ \text{ and } \ \underline{\alpha}_{f|_W}(x) = \underline{\alpha}_{f}(x) 
		\] 
		for any $x \in W(\qb)\cap X_f(\qb)$. 
	\end{lemma}
	
	\begin{proof} Let $\iota : W \hookrightarrow X$ be the inclusion. For an ample Cartier divisor $H$ on $X$, the restriction $H|_W$ is also ample and $h_{H|_W} = h_H \circ \iota$. So the assertion follows.
	\end{proof}

	\section{Several auxiliary results}\label{sec_lemmas}
	
	In this section, we prepare core lemmas for the proof of main results.
	Proposition \ref{prop_hinq} is the key inequality between height functions, 
	and it follows from bigness of certain divisors proven in
	Proposition \ref{prop_bigness}.
	
	We use the following lemma from \cite{Wa22}.
	
	\begin{lemma}[{\cite[Lemma 3.1]{Wa22}}]\label{lem_tech} 
		Let $X$ be a smooth projective variety over $\qb$ and $g: X \dashrightarrow X$
		a dominant rational self-map.
		Fix an ample divisor $H$ on $X$ and its associated height function $h_H$.
		Suppose $d_1, d_2, \zeta \in \RR$ are constants satisfying 
		\begin{equation}\label{eq_tech_cond} 
			d_1, d_2 > 0, \ \ \text{and} \ \ \zeta > \frac{1}{d_1} + \frac{1}{d_2} = \frac{d_1 + d_2}{d_1d_2}.
		\end{equation}
		Assume that there is a non-empty open subset $U \subset X$ and $C > 0$ such that
		\[ \frac{h_H(g^2(z))}{d_1} + \frac{h_H(z)}{d_2} \geq \zeta h_H(g(z)) - C
		\]
		holds for all $z \in X_g(\qb) \cap U(\qb)$.
		Then for a point $y \in X_g(\qb)$ with $\OO_g(y) \subset U(\qb)$, we have
		\[ \underline{\alpha}_g(y) \geq \frac{\zeta d_1}{2}. 
		\] 
	\end{lemma}

	\begin{proposition}\label{prop_xie} Let $X$ be a smooth projective variety defined over $\qb$, and let $f: X \dashrightarrow X$ be a dominant rational self-map of infinite order. Let $U \subset X$ be a non-empty Zariski open subset. Then there exists a $\qb$-point $x \in U(\qb) \cap X_{f}(\qb)$ such that 
		\begin{itemize}
			\item $\OO_f(x) \subset U(\qb)$, 
			
			\item $\OO_f(x)$ is infinite, and 
			
			\item $(X, f, x)$ has the DML property.
		\end{itemize} 
	\end{proposition}

	\begin{proof} This is proved in the proof of \cite[Proposition 3.27]{Xi22}. 
		For the completeness, we sketch how to deduce the statements formally.
		We use the terminologies from \cite[Section 3]{Xi22}.
		First note the $f|_U \colon U \dashrightarrow U$ is also a dominant rational map.
		By \cite[Proposition 3.24]{Xi22} and the proof of \cite[Proposition 3.27]{Xi22}, there is a non-empty adelic open subset
		$A \subset U(\overline{\QQ})$ and a positive integer $m \geq 1$ such that for all $x \in A$ we have
		\begin{itemize}
			\item[$(1)$ ] $x \in U_{f|_U}(\overline{\QQ})$; 
			\item[$(2)$ ] Zariski closure of $O_{f|_{U}^m}(x)$ in $U$ is irreducible; 
			\item[$(3)$ ] $(U,f|_U, x)$ has the DML property.
		\end{itemize}
		This implies for all $x \in A$ we have
		\begin{itemize}
			\item[$(1')$ ] $x \in X_{f}(\overline{\QQ})$ and $\OO_f(x) \subset U$; 
			\item[$(2')$ ] Zariski closure of $O_{f^m}(x)$ in $X$ is irreducible; 
			\item[$(3')$ ] $(X,f, x)$ has the DML property.
		\end{itemize} 
		By the item $(2')$, if $x \in A$ has finite $f$-orbit,
		then it is fixed by $f^m$.
		Since $f$ has infinite order, $f^m$ is not the identity and hence
		$\Gamma_{f^m} \cap \Delta \subsetneq \Delta = X$ 
		where $\Gamma_{f^m}$ is the graph of
		$f^m$ and $\Delta$ is the diagonal of $X \times X$.
		Since $A$ is Zariski dense in $X$ (by \cite[Proposition 3.18.(i) and Remark 3.20]{Xi22}), $A \not\subset \Gamma_{f^m} \cap \Delta$
		and any elements of $A \setminus \Gamma_{f^m} \cap \Delta$
		are what we wanted.
	\end{proof}
	
	\begin{remark}\label{rmk:adelic-open}
		By the proof, there is a non-empty adelic open subset 
		$A_0 \subset U(\overline{\QQ})$ such that the statement holds for all $x \in A_0$.
	\end{remark}
	
	\begin{proposition}\label{prop_hinq} 
		Let $X$ be a smooth projective variety of dimension $N$ defined over $\qb$ and 
		$f: X \dashrightarrow X$ a dominant rational self-map. 
		Let $p \in \{1, \dots, N\}$ be a positive integer such that 
		\begin{align*}
			\delta_{p}(f) = \max\{\delta_{1}(f),\dots, \delta_{N}(f)\}.
		\end{align*}
		Fix an ample Cartier divisor $H$ on $X$ and associated height function $h_H$. 
		Then for any $\varepsilon \in (0, 1)$, there exists $n_{\varepsilon}\in \ZZ_{>0}$
		with the following property: for each $n \geq n_{\varepsilon}$,
		there is a non-empty Zariski open subset
		$U_n \subset X$ and a constant $C_n > 0$ such that
		\[ \frac{h_H(f^{2n}(z))}{ \delta_{p+1}(f)^{n}} + \frac{h_H(z)}{ \delta_{p-1}(f)^{n}} \geq \frac{\varepsilon^{n} \delta_{p}(f)^{n} h_H(f^{n}(z))}{ \delta_{p-1}(f)^{n} \delta_{p+1}(f)^{n}} - C_n
		\]
		holds for all $z\in X_{f^n}(\qb) \cap U_n(\qb)$.
	\end{proposition}

	\begin{proof} 
		Fix $\varepsilon \in (0, 1)$. For each positive integer $n$, we fix a commutative diagram
		\[
		\xymatrix{ & & &\tx \ar[dl]_{\widetilde{\psi}} \ar[dr]^{\widetilde{\phi}} \ar@/_25pt/[lldd]_{\psi} \ar@/^25pt/[rrdd]^{\phi} \ar[dd]^(.3){\pi}& & \\ 
			&  &X' \ar@{-->}[rr] \ar[dl]_{\psi'} \ar[dr]^{\phi'} & &X' \ar[dl]_{\psi'} \ar[dr]^{\phi'} & \\
			& X \ar@{-->}[rr]_{f^{n}} & & X \ar@{-->}[rr]_{f^{n}} & & X }
		\]
		where $X', \widetilde{X}$ are smooth projective varieties,
		$\psi', \widetilde{\psi}$ are birational morphisms, 
		and $\phi', \widetilde{\phi}$ are morphisms. Set 
		\begin{align*}
			D_{n} := \frac{1}{ \delta_{p+1}(f)^{n}} \phi^{*}H + \frac{1}{ \delta_{p-1}(f)^{n}} \psi^{*}H - \frac{\varepsilon^{n} \delta_{p}(f)^{n}}{ \delta_{p-1}(f)^{n} \delta_{p+1}(f)^{n}} \pi^{*}H, 
		\end{align*}
		which is an $\RR$-divisor on $\tx$. By Proposition \ref{prop_bigness} below, there exists a positive integer $n_{\varepsilon}$ such that the $\RR$-divisor $D_n$ is big for every $n \geq n_{\varepsilon}$. 
		
		Fix an arbitrary positive integer $n \geq n_{\varepsilon}$. Let $\B_{+}(D_n)$ denote the augmented base locus of $D_n$. Since $D_n$ is big, $\B_{+}(D_n) \subsetneq \tx$ is a proper Zariski closed subset. We set 
		\[ U_n = X \setminus \pi(\B_{+}(D_n)),  
		\]
		which is a non-empty Zariski open subset of $X$. 
		
		For each $z\in X(\qb)_{g} \cap U_n(\qb)$, we take $\widetilde{z} \in \tx(\qb)$ such that $\psi(\widetilde{z}) = z$. From the functoriality of height functions, it follows that 
		\[ h_H(f^{2n}(z)) = h_H(\phi(\widetilde{z})) = h_{\tx, \phi^{\ast}H}(\widetilde{z}) + O(1), 
		\]
		\[ h_H(z) = h_H(\psi(\widetilde{z})) = h_{\tx, \psi^{\ast}H}(\widetilde{z}) + O(1), 
		\]
		and 
		\[ h_H(f^n(z)) = h_{\tx, \pi^{\ast}H}(\widetilde{z}) + O(1). 
		\]
		On the other hand, by the definition of $D_n$, we have 
		\[ \frac{h_{\tx, \phi^{\ast}H}(\widetilde{z}) }{ \delta_{p+1}(f)^{n}} + \frac{h_{\tx, \psi^{\ast}H}(\widetilde{z})}{ \delta_{p-1}(f)^{n}} = \frac{\varepsilon^{n} \delta_{p}(f)^{n} h_{\tx, \pi^{\ast}H}(\widetilde{z})}{ \delta_{p-1}(f)^{n} \delta_{p+1}(f)^{n}} + h_{\tx, D_n}(\widetilde{z}) + O(1). 
		\]
		Therefore, 
		\[ \frac{h_H(f^{2n}(z))}{ \delta_{p+1}(f)^{n}} + \frac{h_H(z)}{ \delta_{p-1}(f)^{n}} = \frac{\varepsilon^{n} \delta_{p}(f)^{n} h_H(f^{n}(z))}{ \delta_{p-1}(f)^{n} \delta_{p+1}(f)^{n}} + h_{\tx, D_n}(\widetilde{z}) + O(1), 
		\]
		where these $O(1)$ are independent of the choice of $z$. We notice that $\widetilde{z} \notin \B_{+}(D_n)$ as $z \notin \pi(\B_{+}(D_n))$. Then by \cite[Lemma 2.26(1)]{LS21}, 
		\[ \frac{h_H(f^{2n}(z))}{ \delta_{p+1}(f)^{n}} + \frac{h_H(z)}{ \delta_{p-1}(f)^{n}} \geq \frac{\varepsilon^{n} \delta_{p}(f)^{n} h_H(f^{n}(z))}{ \delta_{p-1}(f)^{n} \delta_{p+1}(f)^{n}} - C_n
		\]
		holds for all $z\in X_{f^n}(\qb) \cap U_n(\qb)$, where $C_n$ is independent of the choice of $z$. Clearly we may choose $C_n > 0$. This completes the proof. 
	\end{proof}

	The following criterion about the bigness of $\RR$-divisors was used in the proof of Proposition \ref{prop_hinq}. 
	
	\begin{proposition}\label{prop_bigness} 
		Let $X$ be a smooth projective variety of dimension $N$ defined over
		an algebraically closed field of characteristic zero, and let $f: X \dashrightarrow X$ be a dominant rational self-map. 
		Let $p \in \{1, \dots, N\}$ be a positive integer such that 
		\begin{align*}
			\delta_{p}(f) = \max\{\delta_{1}(f),\dots, \delta_{N}(f)\}.
		\end{align*}
		Fix $\varepsilon \in (0,1)$ and fix an ample Cartier divisor $H$ on $X$. 
		For each positive integer $n$, we fix a commutative diagram
		\[
		\xymatrix{ & & &\tx \ar[dl]_{\widetilde{\psi}} \ar[dr]^{\widetilde{\phi}} \ar@/_25pt/[lldd]_{\psi} \ar@/^25pt/[rrdd]^{\phi} \ar[dd]^(.3){\pi}& & \\ 
			&  &X' \ar@{-->}[rr] \ar[dl]_{\psi'} \ar[dr]^{\phi'} & &X' \ar[dl]_{\psi'} \ar[dr]^{\phi'} & \\
			& X \ar@{-->}[rr]_{f^{n}} & & X \ar@{-->}[rr]_{f^{n}} & & X }
		\]
		where $X', \widetilde{X}$ are smooth projective varieties,
		$\psi', \widetilde{\psi}$ are birational morphisms, 
		and $\phi', \widetilde{\phi}$ are morphisms. Set 
		\begin{align*}
			D_{n} := \frac{1}{ \delta_{p+1}(f)^{n}} \phi^{*}H + \frac{1}{ \delta_{p-1}(f)^{n}} \psi^{*}H - \frac{\varepsilon^{n} \delta_{p}(f)^{n}}{ \delta_{p-1}(f)^{n} \delta_{p+1}(f)^{n}} \pi^{*}H.
		\end{align*}
		Then there is $n_{\varepsilon} \geq 1$ depending only on $X,f,H$, and $\varepsilon$ such that for all $n \geq n_{\varepsilon}$, $D_{n}$ is a big $\RR$-divisor. 
	\end{proposition}

	\begin{proof} 
		Write $\delta_i = \delta_i(f)$ for ease of notation. It is equivalent to consider the $\RR$-divisor
		\begin{align*}
			D'_{n} := \delta_{p-1}^{n} \delta_{p+1}^{n}D_{n}= \delta_{p-1}^{n}\phi^{*}H + \delta_{p+1}^{n} \psi^{*}H - \varepsilon^{n} \delta_{p}^{n} \pi^{*}H. 
		\end{align*} 
		Note that $D'_{n}$ is big if and only if $\vol(D'_{n}) > 0$, where $\vol$ denotes the volume (see \cite[Section 2.2.C]{La04}). 
		Since $\delta_{p-1}^{n}\phi^{*}H + \delta_{p+1}^{n} \psi^{*}H $ and $\varepsilon^{n} \delta_{p}^{n} \pi^{*}H$ are nef $\RR$-divisors, 
		by \cite[Example 2.2.47]{La04}, we have
		\begin{align*}
			\vol(D'_{n}) \geq \left( \left(\delta_{p-1}^{n}\phi^{*}H + \delta_{p+1}^{n} \psi^{*}H \right)^{N}\right)
			- N\left( \left(\delta_{p-1}^{n}\phi^{*}H + \delta_{p+1}^{n} \psi^{*}H \right)^{N-1}\cdot \varepsilon^{n} \delta_{p}^{n} \pi^{*}H \right).
		\end{align*}
		First note that
		\begin{align*}
			\left( \left(\delta_{p-1}^{n}\phi^{*}H + \delta_{p+1}^{n} \psi^{*}H \right)^{N}\right) &\geq 
			\left(\delta_{p-1}^{p} \delta_{p+1}^{N-p}\right)^{n} \left( \phi^{*}H^{p}\cdot \psi^{*}H^{N-p}\right)\\
			& = \left(\delta_{p-1}^{p} \delta_{p+1}^{N-p}\right)^{n} \deg_{p,H}(f^{2n}) =: L_{n}.
		\end{align*}
		Next we have 
		\begin{align*}
			&N\left( \left(\delta_{p-1}^{n}\phi^{*}H + \delta_{p+1}^{n} \psi^{*}H \right)^{N-1}\cdot \varepsilon^{n} \delta_{p}^{n} \pi^{*}H \right)\\
			= \ &\varepsilon^{n}N \delta_{p}^{n} \sum_{i=0}^{N-1} \binom{N-1}{i} \left( \delta_{p-1}^{i} \delta_{p+1}^{N-1-i}\right)^{n} (\phi^{*}H^{i}\cdot \psi^{*}H^{N-1-i} \cdot \pi^{*}H).
		\end{align*}
		
		By \cite[Corollary 3.4.5]{Da20}, the class 
		\begin{align*}
			(N-i+1)^{i} \frac{(\phi^{*}H^{i}\cdot \pi^{*}H^{N-i})}{(\pi^{*}H^{N})} \pi^{*}H^{i} - \phi^{*}H^{i}
		\end{align*}
		is pseudo-effective for $i \in \{0,\dots,N-1\}$. Since $\psi^{*}H^{N-1-i} \cdot \pi^{*}H$ is a nef class (see \cite[Sections 3.1 and 3.3]{Da20}), we have 
		\begin{align*}
			&(\phi^{*}H^{i} \cdot \psi^{*}H^{N-1-i} \cdot \pi^{*}H) \\
			\leq \ &(N-i+1)^{i} \frac{(\phi^{*}H^{i}\cdot \pi^{*}H^{N-i})}{(\pi^{*}H^{N})} (\pi^{*}H^{i} \cdot \psi^{*}H^{N-1-i} \cdot \pi^{*}H)\\
			= \ &(N-i+1)^{i} \frac{\deg_{i,H}(f^{n})}{(H^{N})} \deg_{i+1,H}(f^{n}) \\
			\leq \ &\frac{(N+1)^{N-1} }{(H^{N})} \deg_{i,H}(f^{n}) \deg_{i+1,H}(f^{n}).  
		\end{align*}
		Thus we obtain
		\begin{align*}
			&\varepsilon^{n}N \delta_{p}^{n} \sum_{i=0}^{N-1} \binom{N-1}{i} \left( \delta_{p-1}^{i} \delta_{p+1}^{N-1-i}\right)^{n} 
			(\phi^{*}H^{i}\cdot \psi^{*}H^{N-1-i} \cdot \pi^{*}H)\\
			\leq \ &\varepsilon^{n} N^{2}2^{N-1} \frac{(N+1)^{N-1} }{(H^{N})} \delta_{p}^{n}
			\max_{0 \leq i \leq N-1}\left\{ \left( \delta_{p-1}^{i} \delta_{p+1}^{N-1-i}\right)^{n}  \deg_{i,H}(f^{n}) \deg_{i+1,H}(f^{n})  \right\}.
		\end{align*}
		Let us write
		\begin{align*}
			R_{n} = N^{2}2^{N-1} \frac{(N+1)^{N-1} }{(H^{N})} \delta_{p}^{n}
			\max_{0 \leq i \leq N-1}\left\{ \left( \delta_{p-1}^{i} \delta_{p+1}^{N-1-i}\right)^{n}  \deg_{i,H}(f^{n}) \deg_{i+1,H}(f^{n})  \right\}.
		\end{align*}
		Then it follows that  
		\begin{align*}
			\lim_{n\to \infty} L_{n}^{1/n} &= \delta_{p-1}^{p} \delta_{p+1}^{N-p} \delta_{p}^{2}, \\
			\lim_{n\to \infty} R_{n}^{1/n} &= \delta_{p}\max_{0 \leq i \leq N-1}\left\{ \delta_{p-1}^{i} \delta_{p+1}^{N-1-i} \delta_{i} \delta_{i+1} \right\}, 
		\end{align*}
		where for the latter we use Lemma \ref{lem_limax} below. 
		By Lemma \ref{lemma:ineq-of-dynamical-degrees}, we conclude that  
		\begin{align*}
			\lim_{n \to \infty}R_{n}^{1/n} \leq \lim_{n \to \infty}L_{n}^{1/n}.
		\end{align*}
		Thus there is $n_{\varepsilon} \geq 1$ such that for all $n \geq n_{\varepsilon}$ we have
		\begin{align*}
			L_{n} > \varepsilon^{n} R_{n}.
		\end{align*}
		This implies $\vol(D'_{n}) > 0$, or equivalently, the $\RR$-divisor
		\[ D_n = \frac{D'_n}{\delta^n_{p-1} \delta^n_{p+1}}
		\]
		is big for all $n \geq n_{\varepsilon}$.
	\end{proof}

	The following lemma is elementary, so we omit the proof here. 
	
	\begin{lemma}\label{lem_limax} For each $i = 1, \dots, k$, let $\{ a_{i, n} \}_{n\in \ZZ_{>0}}$ be a sequence of positive numbers such that the limit $a_i = \lim\limits_{n\to \infty} a_{i, n}^{1/n}$
		exists. Then the limit 
		\[ \lim_{n\to \infty} \max_{1\leq i\leq k} \{ a_{i, n} \}^{1/n}
		\]
		exists as well and equals $\max_{i} \{ a_i\}$. 
	\end{lemma}

	\section{Proofs of main theorems}\label{sec_prf}
	
	\subsection{Lower bounds of arithmetic degrees}
	
	\begin{theorem}[Theorem \ref{thm_lb_pch}] 
		Let $f: X \dashrightarrow X$ be a dominant rational self-map of a smooth projective variety $X$ defined over $\qb$. Assume that $f$ is $p$-cohomologically hyperbolic. 
		
		\smallskip $(1)$ For every $x \in X_f(\qb)$ with generic orbit $\OO_f(x)$, we have 
		\[ \underline{\alpha}_f(x) \geq \frac{\delta_p(f)}{\delta_{p-1}(f)}. 
		\]
		
		\smallskip $(2)$ There exists a sequence $\{x_i\}_{i\geq 1} \subset X_f(\qb)$ of $\qb$-points, such that $(X, f, x_i)$ has the DML property for each $i$, and the sequence $\{\underline{\alpha}_f(x_i)\}_{i\geq 1}$ of lower arithmetic degrees satisfies 
		\[ \lim_{i\to \infty} \underline{\alpha}_f(x_i) \geq \frac{\delta_p(f)}{\delta_{p-1}(f)}.
		\]
	\end{theorem}
	
	\begin{proof} Fix an ample Cartier divisor $H$ on $X$. We write $\delta_i = \delta_i(f)$ and set 
		\[ \gamma = \max \left\{ \frac{\delta_{p-1}}{\delta_p}, \frac{\delta_{p+1}}{\delta_p}\right\} < 1. 
		\]
		Then for every $\varepsilon \in (\gamma, 1)$, there exists a positive integer $n'_{\varepsilon}$ such that for every $n\geq n'_{\varepsilon}$, 
		\[ \varepsilon^n \delta_p^n > 2\delta_{p-1}^n, \ \text{ and } \ \varepsilon^n \delta_p^n > 2\delta_{p+1}^n, 
		\]
		and thus, 
		\[ \frac{\varepsilon^n \delta_p^n}{\delta_{p-1}^n \delta_{p+1}^n} > \frac{1}{\delta_{p-1}^n} + \frac{1}{\delta_{p+1}^n}. 
		\]
		
		By Proposition \ref{prop_hinq}, for every $\varepsilon \in (\gamma, 1)$, there exist $n_{\varepsilon}\in \ZZ_{>0}$, such that for each $n \geq n_{\varepsilon}$, 
		\[ \frac{h_H(f^{2n}(z))}{ \delta_{p+1}^{n}} + \frac{h_H(z)}{ \delta_{p-1}^{n}} \geq \frac{\varepsilon^{n} \delta_{p}^{n} h_H(f^{n}(z))}{ \delta_{p-1}^{n} \delta_{p+1}^{n}} - C_n
		\]
		holds for all $z\in X_{f^n}(\qb) \cap U_n(\qb)$, where $U_n\subset X$ is a non-empty Zariski open subset of $X$ and $C_n > 0$ is a constant, both are independent of $z$. 
		
		\medskip $(1)$ We fix $\varepsilon \in (\gamma, 1)$ and fix an arbitrary $n\geq \max\{ n_{\varepsilon}, n'_{\varepsilon}\}$. 
		
		Since $\OO_f(x)$ is generic by assumption, there is $j_0 \geq 0$ such that
		$f^j(x) \in U_n(\qb)$ for all $j \geq j_0$. Then by Lemma \ref{lem_tech}, 
		\[ \underline{\alpha}_{f^n}(f^{j_0}(x)) \geq \frac{\varepsilon^n \delta_p^n}{2\delta_{p-1}^n}. 
		\]
		Since $\underline{\alpha}_{f}(x)^n = \underline{\alpha}_{f^n}(x) = \underline{\alpha}_{f^n}(f^{j_0}(x))$ by Lemma \ref{lem_ks13}, 
		\[ \underline{\alpha}_{f}(x) \geq \frac{\varepsilon \delta_p}{2^{1/n}\delta_{p-1}},  
		\]
		and thus, letting $n\to +\infty$, 
		\[ \underline{\alpha}_{f}(x) \geq \frac{\varepsilon \delta_p}{\delta_{p-1}}.  
		\]
		Therefore, we have 
		\[ \underline{\alpha}_{f}(x) \geq \frac{\delta_p}{\delta_{p-1}}  
		\]
		as $\varepsilon \in (\gamma, 1)$ is arbitrary. 
		
		\medskip $(2)$ We set $\varepsilon = 1 - 1/k$ where $k$ is an integer.
		If $k$ is large, we have $ \varepsilon \in (\gamma, 1)$.
		Let $n\geq \max\{ n_{\varepsilon}, n'_{\varepsilon}\}$. 
		By Proposition \ref{prop_xie}, there exists a $\qb$-point $x_{k, n} \in U_n(\qb) \cap X_{f}(\qb)$ such that $\OO_f(x_{k, n}) \subset U_n(\qb)$, $\OO_f(x_{k, n})$ is infinite, and $(X, f, x_{k, n})$ has the DML property. Similarly by Lemma \ref{lem_tech}, we obtain 
		\[ \underline{\alpha}_{f}(x_{k, n}) \geq \frac{(1- 1/k) \delta_p}{2^{1/n}\delta_{p-1}}
		\]
		for large $k$.
		Hence
		\[ \sup_{k, n} \left\{ \underline{\alpha}_{f}(x_{k, n}) \right\} \geq \frac{\delta_p}{\delta_{p-1}}, 
		\]
		and we can find a sequence $\{x_i\}_{i\geq 1} \subset X_f(\qb)$ of $\qb$-points among $\{ x_{k, n}\}$ such that 
		\[ \lim_{i\to \infty} \underline{\alpha}_f(x_i) \geq \frac{\delta_p}{\delta_{p-1}}, 
		\]
		and $(X, f, x_i)$ has the DML property for each $i$.	
	\end{proof}
	
	\begin{remark}
		In the proof of (2), there is actually a non-empty adelic open subset $A_{k,n} \subset U_n(\qb)$ such that for all $x \in A_{k,n}$ we have
		\begin{align*}
			\underline{ \alpha}_f(x) \geq \frac{(1- 1/k) \delta_p}{2^{1/n}\delta_{p-1}}.
		\end{align*}
		This follows from Remark \ref{rmk:adelic-open}.
		Therefore, we have shown that under the assumption of
		Theorem \ref{thm_lb_pch}, for arbitrary $\epsilon>0$, there is a non-empty adelic open subset $A \subset X(\qb)$ such that
		\begin{align*}
			\underline{ \alpha}_f(x) \geq \frac{\delta_p}{\delta_{p-1}} -\epsilon
		\end{align*}
		for all $x \in A$.
	\end{remark}

	\begin{remark} 
		Although the cohomological hyperbolicity is only used to ensure the assumption (\ref{eq_tech_cond}) in Lemma \ref{lem_tech}, this technical condition is indispensable in our argument. Even for a dominant rational self-map $f$ of $\PP^2$ with $\delta_1(f) = \delta_2(f) > 1$, currently we are unable to find a $\qb$-point whose lower arithmetic degree is strictly larger than $1$. One perhaps expects the existence of invariant fibrations (for some iterate of $f$) in this case. Unfortunately, this is not true in general; see \cite{KPR16}. 
	\end{remark}

	Corollary \ref{thm_lb_onech} about $1$-cohomologically hyperbolic self-maps follows immediately. 
	
	\begin{proof}[Proof of Corollary \ref{thm_lb_onech}] Note that $\delta_0(f) = 1$. Claim $(2)$ then follows from Theorem \ref{thm_lb_pch}(2) immediately. Moreover, Theorem \ref{thm_lb_pch}(1) says that $\underline{\alpha}_f(x) \geq \delta_1(f)$. As we have $\underline{\alpha}_f(x) \leq \overline{\alpha}_f(x) \leq \delta_1(f)$ in general, Claim $(1)$ follows. 
	\end{proof}

	Combining with the dynamical Mordell--Lang conjecture for \'etale morphisms \cite{BGT10}, we obtain the following result.

	\begin{corollary} Let $f: X \dashrightarrow X$ be a $1$-cohomologically hyperbolic dominant rational self-map of a smooth projective variety $X$ defined over $\qb$. Assume that $f$ is an extension of an \'etale dominant morphism $f^{\circ}: X^{\circ} \to X^{\circ}$ where $X^{\circ}\subset X$ is a non-empty Zariski open subset. Then for every $x\in X_f(\qb)$ with Zariski dense orbit $\OO_f(x)$, $\alpha_f(x)$ exists and $\alpha_f(x) = \delta_1(f)$. 
	\end{corollary}
	
	\begin{proof} By Corollary \ref{thm_lb_onech}(1), it is enough to show that $\left\{ n \in \ZZ_{\geq 0} : f^n(x) \in Z \right\}$ is finite for every proper closed subvariety $Z \subset X$. This is because $f^{n_0}(x) \in X^{\circ}$ for some positive integer $n_0$ and $\left\{ n \in \ZZ_{\geq 0} : f^n(f^{n_0}(x)) \in Z\cap X^{\circ} \right\}$ is finite by \cite[Corollary 1.4]{BGT10}. 
	\end{proof}

	\begin{remark} 
    
    Let $X$ be a quasi-projective variety over an algebraically closed field of characteristic zero with non-negative logarithmic Kodaira dimension (see \cite[Page 326]{Ii82} for the definition). By \cite[Theorem 11.7]{Ii82}, any proper dominant self-morphism of $X$ is \'etale. 
    For example, any dominant monomial map $f$ on the algebraic torus $\mathbb{G}_{m}^n$
    is \'etale. By \cite[Cororally B]{FaWu12} or \cite[Theorem 1]{Li12}, 
    $f$ is $1$-cohomologically hyperbolic if and only if
    the characteristic polynomial of its defining matrix 
    is the minimal polynomial of a Pisot number 
    (i.e.\ a real algebraic integer larger than $1$ whose Galois conjugates
    have modulus less than $1$). See also Example \ref{ex:auto_A3_coh_hyp} below for an explicit example. 
	\end{remark}

	\subsection{Existence of Zariski dense orbits}
	
	Now we show the existence of Zariski dense orbits. The following argument is inspired by \cite[Section 8]{JSXZ21}. 
	
	Let $f \colon X \dashrightarrow X$ be a $p$-cohomologically hyperbolic 
	dominant rational map on a smooth projective variety $X$ defined over $\qb$.
	Let $N = \dim X$.
	By the log concavity of dynamical degrees (\ref{eq_logc}), we have 
	\begin{equation}\label{eq_zdodd}
		\delta_{1}(f),\dots, \delta_{p}(f) \geq \frac{\delta_{p}(f)}{\delta_{p-1}(f)} \ \ \text{and} \ \ \delta_{p+1}(f)>\delta_{p+2}(f)>\cdots > \delta_{N}(f).
	\end{equation}

	\begin{proposition}\label{prop:exist-zdo-general-case}
		Let $X$ be a smooth projective variety defined over $\qb$ of dimension $N$.
		Let $f \colon X \dashrightarrow X$ be a $p$-cohomologically hyperbolic dominant rational map satisfying 
		\[ \delta_{N}(f) < \frac{\delta_{p}(f)}{\delta_{p-1}(f)}. 
		\]
		Set
		\begin{align*}
			l_0 := \min \left\{ l \geq 0 \,\middle|\, \delta_{2+l}(f) < \frac{\delta_{p}(f)}{\delta_{p-1}(f)} \right\}. 
		\end{align*}
		Then there is a non-empty adelic open subset $A \subset X_f(\qb)$ such that for all $x \in A$, 
		\begin{itemize} 
			\item $(X, f, x)$ has the DML property, and 
			\item $\codim_X \overline{\OO_f(x)} \leq l_0$. 
		\end{itemize}
		Here $\overline{\OO_f(x)}$ is the closure of $\OO_f(x)$ with respect to Zariski topology.
		
		\smallskip In particular, if $p=1$, then $l_0=0$ and $\OO_f(x)$ is Zariski dense for all $x \in A$.
		Moreover, in this case, we have $\alpha_f(x) = \delta_{1}(f)$ for all $x \in A$.
	\end{proposition}
	\begin{proof} 
		
		Let us fix $\epsilon > 0$ such that
		\begin{align*}
			\delta_{2+l_0}(f) < \frac{\delta_{p}(f)}{\delta_{p-1}(f)} - \epsilon.
		\end{align*}
		Take a non-empty Zariski open subset $U \subset X$ such that
		$f|_U \colon U \longrightarrow X$ is a morphism and quasi-finite.
		Then by Remark \ref{rmk:existence-points-adelic-open-ver},
		Remark \ref{rmk:adelic-open},  and \cite[Proposition 3.18(v)]{Xi22},
		there is a non-empty adelic open subset $A \subset X(\qb)$ such that
		\begin{itemize}
			\item[(i) ] $A \subset X_f(\qb)$, 
			\item[(ii) ] $\OO_f(x) \subset U(\qb)$ for all $x \in A$, 
			\item[(iii) ] $(X, f, x)$ has the DML property, 
			\item[(iv) ] $\underline{\alpha}_f(x) \geq \delta_{p}(f)/\delta_{p-1}(f) - \epsilon$.
		\end{itemize}
		For $x \in A$, there is an irreducible component $Z$ of $\overline{\OO_f(x)}$
		with the generic point $\eta \in Z$ satisfying the following properties:
		\begin{itemize} 
			\item[$(1)$ ] $\dim Z = \dim \overline{\OO_f(x)}$;
			\item[$(2)$ ] $\eta \notin I(f^n)$ for all $n \geq 0$;
			\item[$(3)$ ] $f^k(\eta) = \eta$ for some $k \geq 1$.
		\end{itemize}
		Since $\OO_f(x) \subset U$, we have $\OO_f(\eta) \subset U$.
		Thus $f^k$ is quasi-finite at $\eta$.
		Consider the following diagram
		\begin{align*}
			\xymatrix@C=50pt{
				Z \ar@{-->}[r]^{f^k|_Z =:g } \ar@{}[d]|{\bigcap} & Z \ar@{}[d]|{\bigcap} \\
				X \ar@{-->}[r]_{f^k} & X.
			}
		\end{align*}
		Pick $s \geq 0$ such that $f^s(x) \in Z$.
		Then we have 
		\[ \left(\frac{ \delta_{p}(f) }{ \delta_{p-1}(f) } - \epsilon \right)^k \leq \underline{\alpha}_f(x)^k = \underline{\alpha}_{f^k}(f^s(x)) =\underline{\alpha}_g(f^s(x)) \leq \delta_{1}(g), 
		\]
		where the two equalities follow from Lemmas \ref{lem_ks13} and \ref{lem_mmsz2.5}, the first inequality follows from the item (iv), and the last inequality follows from Proposition \ref{prop_jsxz3.11}. On the other hand, by Lemma \ref{lem:comp-dyn-deg}, we have  
		\[ \delta_{1}(g) \leq \delta_{1+c}(f^k) = \delta_{1+c}(f)^k, \]
		where 
		\[ c = \codim_X Z. \] 
		It then follows that 
		\begin{align*}
			\delta_{2+l_0}(f) < \frac{\delta_{p}(f)}{\delta_{p-1}(f)} - \epsilon
			\leq \delta_{1+c}(f). 
		\end{align*}
		By (\ref{eq_zdodd}) and the definition of $l_0$, we obtain
		\begin{align*}
			1 + \codim_X Z = 1 + c < 2 + l_0,
		\end{align*}
		that is, $\codim_X \overline{\OO_f(x)} = \codim_X Z \leq l_0$.
		
		Finally, we assume $p=1$.
		Then $\delta_{2}(f) < \delta_{1}(f)$ and we get $l_0=0$.
		Therefore $\OO_f(x)$ is Zariski dense for all $x \in A$.
		Since $(X,f,x)$ has the DML property for all $x \in A$,
		$\OO_f(x)$ is generic. Then by Corollary \ref{thm_lb_onech}(1), we have $\alpha_f(x) = \delta_{1}(f)$. 
	\end{proof}

	Theorem \ref{thm_zdo} follows from Proposition \ref{prop:exist-zdo-general-case} immediately. \qed

	\begin{proof}[Proof of Theorem \ref{thm_zdo_three}] By assumption, $\delta_1(f) \neq \delta_2(f)$. If $\delta_1(f) > \delta_2(f)$, then the claim follows from Theorem \ref{thm_zdo}. 
		
		Now we assume $\delta_1(f) < \delta_2(f)$. Since $f$ does not admit invariant rational functions,
		by \cite[Corollary 3.3]{Ca10} (see also \cite[Corollary 1.3]{BMT22}), there are only finitely many totally $f$-invariant hypersurfaces. Let $Z$ be the union of those hypersurfaces.
		By Propositions \ref{prop:exist-zdo-general-case} and \ref{prop_xie}, there exists $x\in X_f(\qb) \setminus Z(\qb)$ such that $\OO_f(x) \subset X_f(\qb) \setminus Z(\qb)$ and $\codim_X \overline{\OO_f(x)} \leq 1$. It then follows that the orbit $\OO_f(x)$ is Zariski dense in $X$.
	\end{proof}

	\begin{proof}[Proof of Theorem \ref{thm_trans_two}] 
		
		By \cite[Main Theorem and 3rd paragraph in Page 196]{BDJ20} (see also Example \ref{ex_bdj} below), there exists a dominant rational self-map $f: \PP^2 \dashrightarrow \PP^2$ defined over $\qb$ whose dynamical degrees are 
		\[ \delta_2(f) = 5 < \delta_1(f) = 6.8575574092..., 
		\]
		and the first dynamical degree $\delta_1(f)$ is a transcendental number. The claim then follows from Theorem \ref{thm_zdo}. 
	\end{proof}

	\subsection{Examples} In this subsection, we list some examples of cohomologically hyperbolic dominant rational self-maps.

	\begin{example}\label{ex_bdj} Here let us briefly recall the examples from \cite{BDJ20}. We work over a field of characteristic zero. Fix homogeneous coordinates $[x_0 : x_1 : x_2]$ on $\PP^2$ and use affine coordinates $(y_1, y_2) = (x_1/x_0, x_2/x_0)$ on the affine chart $\{ x_0 \neq 0 \} \cong \bba^2$.
		
		Consider a dominant rational self-map $f: \PP^2 \dashrightarrow \PP^2$ of the form $f = g \circ h$, where 
		\[ h(y_1, y_2) = (y_1^a y_2^b, y_1^{-b} y_2^a)
		\] 
		is a monomial map, and 
		\[ g: [x_0 : x_1 : x_2] \mapsto [x_0(x_1+x_2-x_0) : x_1(x_2+x_0-x_1) : x_2(x_0+x_1-x_2)]
		\]
		is a fixed birational involution conjugate to the standard Cremona involution by a projective linear automorphism. By \cite[Proposition 2.8]{BDJ20}, $\delta_1(f)$ is the unique positive solution to the equation 
		\begin{equation}\label{equat} 
			\sum_{j\geq 0} \deg(h^j) \lambda^{-j} = 1.
		\end{equation} 
		Here the degree $\deg (h^j)$ of the monomial map $h^j$ can be easily computed from the associated matrix (see e.g.,  \cite[Section 1.3]{BDJ20}). 
		Moreover, $\delta_1(f)$ is transcendental if $\zeta^n \notin \RR$ for every positive integer $n$ where $\zeta = a+b\sqrt{-1}$ (\cite[Main Theorem]{BDJ20}). Note that $\delta_2(f) = \delta_2(h) = |\zeta|^2$. If we take $\zeta = 1+2\sqrt{-1}$, then 
		\[ \delta_2(f) = 5 < \delta_1(f) = 6.8575574092...
		\]
		If we take $\zeta = (1+2\sqrt{-1})^2 = -3+4\sqrt{-1}$, then 
		\[ \delta_2(f) = 25 >  \delta_1(f) = 13.4496076817...
		\]
		Here $\delta_1(f)$ can be computed numerically from Equation (\ref{equat}) (see \cite[Page 196]{BDJ20}). 
	\end{example}

	\begin{example}[{\cite[Example 1.11]{GS02}}] Let $P(x, y)$ be a homogeneous polynomial of degree $d \geq 2$. Consider the birational map $f: \PP^3 \dashrightarrow \PP^3$ induced by the automorphism of $\bba^3$
		\[ f^{\circ}: \bba^3 \to \bba^3, \ \ (x, y, z) \mapsto \left( xP(x, y) + z, x^{d+1} + by, x \right) 
		\]
		with $b \neq 0$. Then 
		\[ (f^{\circ})^{-1}: \bba^3 \to \bba^3, \ \ (x, y, z) \mapsto \left( z, b^{-1}\big( y - z^{d+1} \big), x - zP\Big(z, b^{-1}\big( y - z^{d+1} \big) \Big) \right). 
		\]
		If $\deg_y P = d$, then 
		\[ \deg (f) = d + 1, \ \ \text{and} \ \ \deg (f^{-1}) = d^2 + d + 1. 
		\]
		We know that $f^{-1}$ is weakly regular, so $f^{-1}$ is algebraically $1$-stable (see \cite[Definition 2.1 and below]{GS02}). In particular,  
		\[ \delta_1(f^{-1}) = d^2 + d + 1,  
		\]
		and hence 
		\[ \delta_1(f) \leq \deg (f) = d+1 < \delta_1(f^{-1}) = \delta_2(f). 
		\] 
	\end{example}

	\begin{example}[{cf. \cite[Theorem 4.1(3) and (4)]{CG04}}]\label{ex:auto_A3_coh_hyp} Consider birational self-maps $f$ and $g$ of $\PP^3$ which are extensions of automorphisms of $\bba^3$
		\[ f^{\circ}(x, y, z) = \left( \alpha xy + az, y^2 + x, y \right) \ \text{ with } \alpha a \neq 0, 
		\]
		and
		\[ g^{\circ}(x, y, z) = \left( x^2 - xz + y, \beta z, bx \right) \ \text{ with } \beta b \neq 0.  
		\]
		Then 
		\[ \delta_2(f) = \delta_1(f^{-1}) = 3 > 2 = \delta_1(f), 
		\]
		and 
		\[ \delta_1(g) = 2 > \frac{1 + \sqrt{5}}{2} = \delta_1(g^{-1}) = \delta_2(g). 
		\]
	\end{example}

	\begin{example}[cf. \cite{BHT22}] Let $F : \bba^4_{\qb} \to \bba^4_{\qb}$ be the birational morphism given by 
		\[ F(x_1, x_2, x_3, x_4) = (x_2, -x_4, x_1 - x_1x_2^2, -x_3 + x_1x_2x_4). 
		\]
		Let $\varphi : \bba^4_{\qb} \to \bba^1_{\qb}$ be the morphism given by $\varphi(x_1, x_2, x_3, x_4) = x_1x_4 - x_2x_3$. Then the fibres of $\varphi$ are invariant by $f$, that is, $\varphi = \varphi \circ F$. For $c \in \qb$, let \[ Z_c = \varphi^{-1}(c) = \{ x_1x_4 - x_2x_3 = c \} \subset \bba^4_{\qb}
		\]
		be the fibre of $\varphi$ over $c$. Let $X_c$ be the closure of $Z_c$ in $\PP^4_{\qb}$, and hence, in homogeneous coordinates $[x_1 : x_2 : x_3 : x_4 : z]$ of $\PP^4_{\qb}$, $X_c$ is given by $x_1x_4 - x_2x_3 = cz^2$. 
		
		We now take $c \in \qb \setminus \{0\}$ so that $X_c$ and $Z_c$ are smooth. Since $X_c$ is invariant under $F$, we have a birational morphism $f_c = F|_{Z_c} : Z_c \to Z_c$ of the smooth affine quadric threefold $Z_c$, and its extension $\widehat{f}_c : X_c \dashrightarrow X_c$, which is a birational self-map of the smooth projective quadric threefold $X_c$. 
		
		By \cite[Lemma 2.4]{BHT22} (and its proof), $\delta_1(\widehat{f}_c)$ is the largest root of the polynomial $t^3 - t^2 - t - 1$, which is approximately 1.8393. It is conjectured that $\delta_2(\widehat{f}_c)$ is the largest root of the polynomial $2t^3 - 3(t^2-1)-4t$, which is not an algebraic integer and is approximately 2.1108 (\cite[Conjecture 1.5(2)]{BHT22}). Assuming this conjecture holds and applying Theorem \ref{thm_zdo} to the $1$-cohomologically hyperbolic birational self-map $( \widehat{f}_c )^{-1}$, we obtain an example of arithmetic degrees which is not an algebraic integer but an algebraic number. 
	\end{example}

	We end this section with a remark about canonical height functions of birational self-maps in dimension two. 
	
	\begin{remark} Let $X$ be a smooth projective surface defined over $\qb$ and let $f : X \dashrightarrow X$ be a birational self-map of $X$ with $\delta = \delta_1(f) > 1$. By \cite[Theorem D]{JR18}, up to birational conjugation, the limit
		\[ \widehat{h}^{+}(x) := \lim_{i\to \infty} h_{\theta^{+}}(f^i(x))\cdot \delta^{-i}
		\] 
		exists for all $x\in X_f(\qb)$, where $\theta^{+}$ is the unique  nonzero nef class satisfying $f^{\ast} \theta^{+} \equiv \delta \theta^{+}$ up to scaling, and $h_{\theta^+}$ is a Weil height function associated with any divisor from the class. The function $\widehat{h}^+$ is independent of the choice because the difference of $h_{\theta^+}$ is bounded by the square root of an ample height function and the growth of an ample height along the orbit $\{f^i(x)\}$ is bounded by $(\delta + \epsilon)^i$ for any small $\epsilon > 0$ (cf.\ \cite[Theorem 1.4]{Ma20}).
		Moreover, $\widehat{h}^{+}$ is finite and non-negative on $X_f(\qb)$. It is also shown in \cite{JR18} that, if $\widehat{h}^{+}$ is identically zero on $X_f(\qb)$, then $$\lim_{i\to \infty} h_H(f^i(x))\cdot \delta^{-i} = 0$$ for every $x \in X_f(\qb)$ and every ample Cartier divisor $H$. Our result Theorem \ref{thm_zdo} is weaker than the assertion $\widehat{h}^{+}(x) > 0$ for every $x \in X_f(\qb)$ with Zariski dense orbit. It is still unknown whether $\widehat{h}^{+}$ can be identically zero or not. 
	\end{remark}

	\section{Self-maps with Large topological degree}\label{sec_lgtopdeg}

	In this section, we prove Theorems \ref{thm:ad-largetopdeg} and \ref{thm:fin-preper-largetopdeg}. Let us start with some lemmas. Lemma \ref{lem-resolution} below should be well-known to experts. Here we sketch the proof just for the benefit of the reader.

	\begin{lemma}\label{lem-resolution}
		Let $X$ be a normal projective variety over
		an algebraically closed field $\kk$ of characteristic zero.
		Then there is a projective birational morphism
		\begin{align*}
			\pi \colon X' \longrightarrow X
		\end{align*}
		such that $X'$ is a smooth projective variety over $\kk$
		and there is an effective Cartier divisor $E$ on $X'$ with the following properties:
		\begin{itemize}
			\item $-E$ is $\pi$-ample;
			
			\item $\pi_*E=0$.
		\end{itemize}
	\end{lemma}
	\begin{proof}
		There is a sequence of morphisms
		\begin{align*}
			\xymatrix{
				X=X_0 & X_1 \ar[l]_(.35){\nu_1} & \ar[l]_{\nu_2} \cdots & X_r \ar[l]_{\nu_r}
			}
		\end{align*}
		such that $X_r$ is smooth, and each $\nu_i$ is a blow-up along a smooth irreducible closed subvariety $C_i \subset \Sing X_{i-1}$,
		where $\Sing$ stands for the singular locus
		(see for example \cite[Theorem 1.1]{BM08}).
		
		We set $X' = X_r$ and $\pi = \nu_1 \circ \cdots \circ \nu_r$. 
		Let $E_i$ be the exceptional divisor of the blow-up $\nu_i$.
		Then $-E_i$ is $\nu_i$-ample by definition of blow-up, and $\nu_i(E_i) \subset \Sing X_{i-1}$.
		We also have $\nu_i(\Sing X_i) \subset \Sing X_{i-1}$.
		Then there are positive integers $n_1,\dots, n_{r-1}$ such that
		\begin{align*}
			E := E_r + n_{r-1}\nu_r^*E_{r-1} + \cdots + n_1 (\nu_2 \circ \cdots \circ \nu_r)^*E_1
		\end{align*}
		is anti-$\pi$-ample. Since $\pi(\Supp E) \subset \Sing X$ and $\Sing X$ has codimension at least two,
		we get $\pi_{*} E =0 $. 
	\end{proof}
	
	\begin{lemma}\label{lem:excep-antiample}
		Let $\pi \colon X \longrightarrow Y$ be a generically finite surjective morphism between smooth projective varieties over
		an algebraically closed field $\kk$ of characteristic zero.
		Then there is a birational morphism $\nu \colon X' \longrightarrow X$ from a projective variety $X'$ 
		and an effective Cartier divisor $E$ on $X'$ such that 
		\begin{itemize}
			\item $-E$ is $(\pi \circ \nu)$-ample; 
			
			\item $(\pi \circ \nu)_{*}E=0$. 
		\end{itemize}
	\end{lemma}
	\begin{proof}
		Let $X \xrightarrow{ \widetilde{\pi}} \widetilde{Y} \xrightarrow{\mu} Y$ be the Stein factorization of $\pi$, where $\widetilde{\pi}$ is birational and $\mu$ is finite. Note that $\widetilde{Y}$ is normal because it is the normalization of $Y$ in the function field of $X$. 
		By Lemma \ref{lem-resolution}, 
		there is a projective birational morphism $p \colon Z \longrightarrow \widetilde{Y}$
		with the following properties:
		\begin{itemize}
			\item[(i)] $Z$ is a smooth projective variety; 
			
			\item[(ii)] there is an effective Cartier divisor $E_Z$ on $Z$ such that $-E_Z$ is $p$-ample and $p_*E_Z=0$. 
		\end{itemize}
		Take a closed immersion $X \subset \PP^{N}_{\kk}$ into a projective space.
		Consider the rational map
		\[
		\xymatrix{
			Z \ar[r]^p \ar@/^20pt/@{-->}[rrr]^f & \widetilde{Y}  \ar@{-->}[r]^{ \widetilde{\pi}^{-1}}  & X \ar[r]& \PP^{N}_{\kk}.
		}
		\]
		Let $U = Z \setminus I(f)$ and $i \colon U \to Z$ the inclusion.
		Since $I(f)$ has codimension at least two and $Z$ is smooth,
		$i_*(f|_U)^*\OOO(1)$ is invertible.
		This invertible sheaf with the global sections coming from the coordinates of $\PP^N$
		defines a linear system on $Z$ whose base ideal $\mathcal{I}$ has cosupport $I(f)$.
		Now take the blow-up $q \colon X' \longrightarrow Z$ of $Z$ along $\mathcal{I}$.
		Then we get the following diagram:
		\begin{align*}
			\xymatrix{
				X \ar[d]_\pi \ar[rd]^{\widetilde{\pi}} &&&\\
				Y & \widetilde{Y} \ar[l]^{\mu} & Z \ar[l]^p & X' \ar[l]^q \ar[lllu]_{\nu}
			}
		\end{align*}
		where $\nu := \widetilde{\pi}^{-1} \circ p \circ q$.
		The map $\nu$ is a morphism because $q$ is the blow-up along the base ideal.
		Let $F$ be the exceptional divisor of the blow-up $q$.
		Then $-F$ is $q$-ample.
		Since the cosupport of $\mathcal{I}$ has codimension at least two,
		$F$ is $q$-exceptional.
		Since $-F$ is $q$-ample and $-E_Z$ is $p$-ample, there is a positive integer $n$ such that
		$-F - nq^*E_Z$ is $(p\circ q)$-ample.
		We set $E =F + nq^*E_Z $.
		Then $E$ is an effective Cartier divisor on $X'$,  and we have
		\begin{align*}
			(\pi \circ \nu)_*E = \mu_*p_*q_*E = n\mu_*p_*E_Z = 0.
		\end{align*}
		As $-E$ is $(p\circ q)$-ample, it is $(\mu \circ p \circ q)$-ample as well because 
		$\mu$ is finite. Thus $-E$ is $(\pi \circ \nu)$-ample.
	\end{proof}

	\begin{lemma}\label{lem:weaklowerbound}
		Let $X$ be a smooth projective variety defined over $ \overline{\QQ}$ of dimension $N \geq 2$, and let $f \colon X \dashrightarrow X$ be a dominant rational map.
		Let $H$ be an ample Cartier divisor on $X$ and fix an associated height function $h_{H}$.
		Then there are constants $\kappa > 0$, $C \geq 0$, and a closed subset $W \subset X$ of codimension at least two, such that 
		\begin{align*}
			h_{H}^{+}(f(x)) \geq \kappa h_{H}^{+}(x) - C
		\end{align*}
		for all $x \in X ( \overline{\QQ})$ satisfying $x \notin I(f)$ and $f(x) \notin W$. Here we recall that $h^{+}_H = \max\{1, h_H\}$. 
	\end{lemma}  
	
	\begin{proof}
		By Lemma \ref{lem:excep-antiample}, we can take and fix the following commutative diagram
		\[
		\xymatrix{
			&X' \ar[dl]_{q} \ar[rd]^{g}& \\
			X \ar@{-->}[rr]_{f}& & X
		}
		\]
		where $X'$ is a projective variety, $q$ is a birational morphism, and $g$ is a surjective morphism such that there is an effective Cartier divisor $E$ on $X'$ with the following properties: 
		\begin{itemize}
			\item[(i)] $-E$ is $g$-ample; 
			
			\item[(ii)] $g_{*}E = 0$.
		\end{itemize}
		Then there is $l \geq 1$ such that $lg^{*}H - E$ is ample.
		Hence there is $m \geq 1$ such that $m(lg^{*}H - E) - q^{*}H$ is ample.
		Therefore 
		\begin{align*}
			m(lh_{H}^{+}\circ g - h_{E}) - h_{H}^{+}\circ q
		\end{align*}
		is bounded below on $X'( \overline{\QQ})$. Since $E$ is effective, $h_{E}$ is bounded below outside the support $\Supp E$ of $E$.  Hence 
		\begin{align*}
			h_{H}^{+}\circ g  \geq  \frac{1}{ml}h_{H}^{+}\circ q -C
		\end{align*}
		for some $C \geq 0$ on $( X' \setminus \Supp E)( \overline{\QQ})$. 
		
		We set $W = g(\Supp E)$, which has codimension at least two by the item (ii). Let $x \in X( \overline{\QQ})$ satisfying $x \notin I(f)$ and $f(x) \notin W$.
		Pick a point $y \in X'( \overline{\QQ})$ such that $q(y) = x$.
		Then $g(y)=f(x)$ and thus $y \notin \Supp E$.
		Therefore we have 
		\[ h_{H}^{+}(f(x)) = h_{H}^{+}(g(y)) \geq \frac{1}{ml} h_{H}^{+}(q(y)) - C = \frac{1}{ml} h_{H}^{+}(x) - C. \]
		The proof is completed by taking $\kappa = (ml)^{-1}$. 
	\end{proof}
	
	\begin{lemma}\label{lem:bigness-largetopdeg}
		Let $X$ be a smooth projective variety of dimension $N$ defined over an algebraically closed field $\kk$ of characteristic zero.
		Let $f \colon X \dashrightarrow X$ be a dominant rational map.
		Suppose $ \delta_{N}(f) > \delta_{N-1}(f)$ 
		(This is equivalent to $N$-cohomological hyperbolicity).
		Fix an ample Cartier divisor $H$ on $X$.
		For each positive integer $n$, we fix a commutative diagram
		\[
		\xymatrix{ &  X'  \ar[dl]_{\psi'_n} \ar[dr]^{\phi'_n} & \\
			X \ar@{-->}[rr]_{f^{n}} & & X  }
		\]
		where $X'$ is a smooth projective variety, $\psi'_n$ is a birational morphism,
		and $\phi'_n := f^n \circ \psi'_n$ is a morphism.
		Then there is $n_{0} \geq 0$ such that for every $n \geq n_{0}$, the Cartier divisor 
		\begin{align*}
			{\phi'_n}^{*} H - {\psi'_n}^{*}H
		\end{align*}
		is big.
	\end{lemma}
	\begin{proof}
		Since ${\phi'_n}^{*} H$ and ${\psi'_n}^{*}H$ are nef, it is enough to check
		\begin{align*}
			({\phi'_n}^{*} H^{N}) > N({\phi'_n}^{*} H^{N-1}\cdot {\psi'_n}^{*}H)
		\end{align*}
		(cf. \cite[Theorem 2.2.15]{La04}). We have
		\begin{align*}
			&\lim_{n\to \infty} ({\phi'_n}^{*} H^{N})^{1/n} = \lim_{n\to \infty}\deg_{N,H}(f^{n})^{1/n}  
			= \delta_{N}(f);\\
			&\lim_{n\to \infty} \left(N({\phi'_n}^{*} H^{N-1}\cdot {\psi'_n}^{*}H)\right)^{1/n} = \lim_{n\to \infty} \left(N \deg_{N-1,H}(f^{n})\right)^{1/n} = \delta_{N-1}(f).
		\end{align*}
		Since $\delta_{N}(f) > \delta_{N-1}(f)$ by assumption, we are done.
	\end{proof} 
	
	The following argument uses a weak dynamical Mordell--Lang theorem \cite{BGT15, BHS20} whose statement involves the natural density. Let $S$ be a subset of $\mathbb{N}$, the set of all non-negative integers. Recall the \textit{upper natural} (or \textit{asymptotic}) \textit{density} $\overline{d}(S)$ of $S$ is defined by
	\[ \overline{d}(S) := \limsup_{m\to \infty} \frac{\# (S \cap [1,m])}{m}. 
	\]
	Replacing $\limsup$ by $\liminf$ or $\lim$, we define the \textit{lower natural density} $\underline{d}(S)$ or the \textit{natural density} $d(S)$ (if it exists). Note that if $\overline{d}(S) = 0$, then $d(S) = 0$ and $d(\mathbb{N} \setminus S) = 1$. 
	
	\begin{theorem}[Theorem \ref{thm:ad-largetopdeg}]
		Let $X$ be a smooth projective variety of dimension $N$ defined over $\overline{\QQ}$, and let $f \colon X \dashrightarrow X$ be an $N$-cohomologically hyperbolic dominant rational self-map.
		Let $x \in X_{f}( \overline{\QQ})$ with Zariski dense $f$-orbit $\OO_{f}(x)$  in $X$. Assume that 
		\begin{itemize}
			\item[$(\ast)$] the intersection of $\OO_{f}(x)$ and any Zariski closed subset of $X$ of codimension at least two is finite.
		\end{itemize}
		Then we have
		\begin{align*}
			\underline{\alpha}_{f}(x) \geq \frac{ \delta_{N}(f)}{ \delta_{N-1}(f)}.
		\end{align*}
	\end{theorem}
	
	\begin{proof}
		Fix an ample Cartier divisor $H$ on $X$ and fix an associated height function $h_{H}$ so that $h_{H} \geq 1$ (hence $h_{H}^{+}=h_{H}$). 
		
		Write $\delta_i = \delta_i(f)$. Note that $\delta_{N} > 1$. Fix arbitrary $\varepsilon \in (1/\delta_{N}, 1)$.
		For all $n \geq 1$, fix a commutative diagram
		\[
		\xymatrix{ & & &\tx \ar[dl]_{\widetilde{\psi}} \ar[dr]^{\widetilde{\phi}} \ar@/_25pt/[lldd]_{\psi} \ar@/^25pt/[rrdd]^{\phi} \ar[dd]^(.3){\pi}& & \\ 
			&  &X' \ar@{-->}[rr] \ar[dl]_{\psi'} \ar[dr]^{\phi'} & &X' \ar[dl]_{\psi'} \ar[dr]^{\phi'} & \\
			& X \ar@{-->}[rr]_{f^{n}} & & X \ar@{-->}[rr]_{f^{n}} & & X }
		\]
		as in Proposition \ref{prop_bigness}. 
		Then by Proposition \ref{prop_bigness} and Lemma \ref{lem:bigness-largetopdeg}, there is $n_{0} \geq 0$ such that 
		for all $n_{1} \geq n_{0}$ the following two $\RR$-divisors are big:
		\begin{align*}
			& \phi^{*}H + \frac{1}{ \delta_{N-1}^{n_{1}}} \psi^{*}H - \frac{\varepsilon^{n_{1}} \delta_{N}^{n_{1}}}{ \delta_{N-1}^{n_{1}} } \pi^{*}H\\
			&\phi'^{*} H - \psi'^{*}H.
		\end{align*}
		Fix such an $n_{1}$ and set $g = f^{n_{1}}$.
		By the bigness, there is a real number $C \geq 0$ and a proper closed subset $Z \subset X$ such that
		\begin{align}
			\label{ltd:ineq1}&h_{H}(g^{2}(x)) + \frac{1}{ \delta_{N-1}^{n_{1}}} h_{H}(x) - \frac{\varepsilon^{n_{1}} \delta_{N}^{n_{1}}}{ \delta_{N-1}^{n_{1}} } h_{H}(g(x)) \geq -C;\\
			\label{ltd:ineq2}&h_{H}(g(x)) \geq h_{H}(x) -C
		\end{align}
		for all $x \in X_{f}( \overline{\QQ}) \setminus Z( \overline{\QQ})$
		(cf. the proof of Proposition \ref{prop_hinq}).
		By Lemma \ref{lem:weaklowerbound}, there are constants $\kappa >0$, $C' \geq 0$, and a closed subset $W \subset X$ of codimension at least two such that
		\begin{align}\label{ltd:ineq3} 
			h_{H}(g(x)) \geq \kappa h_{H}(x) - C'
		\end{align}
		for all $x \in X_{f}( \overline{\QQ})$ with $g(x) \notin W$.
		
		Now we take a point $x \in X_{f}( \overline{\QQ})$ with Zariski dense $f$-orbit $\OO_f(x)$ and satisfying the assumption $(\ast)$. Then by \cite[Theorem 1.10]{BHS20} (see also \cite[Corollary 1.5]{BGT15}), 
		the set $\left\{ m \geq 0 \ \middle|\ g^{m}(x) \in Z \right\}$ has upper natural density $0$, and hence its complement in $\mathbb{N}$ 
		\begin{align*}
			\Sigma := \left\{ m \geq 0 \ \middle|\ g^{m}(x) \notin Z \right\}
		\end{align*}
		has natural 
		density $1$. By (\ref{ltd:ineq1}) and (\ref{ltd:ineq2}), 
		for all $m \in \Sigma$, we have
		\begin{align*}
			\frac{h_{H}(g^{m+2}(x))}{h_{H}(g^{m+1}(x))} &\geq \frac{\varepsilon^{n_{1}} \delta_{N}^{n_{1}}}{ \delta_{N-1}^{n_{1}} } - \frac{1}{ \delta_{N-1}^{n_{1}}} \frac{h_{H}(g^{m}(x))}{h_{H}(g^{m+1}(x))}  - \frac{C}{h_{H}(g^{m+1}(x))}\\
			& \geq \frac{\varepsilon^{n_{1}} \delta_{N}^{n_{1}}}{ \delta_{N-1}^{n_{1}} } - \frac{1}{ \delta_{N-1}^{n_{1}}} \left(1+\frac{C}{h_{H}(g^{m+1}(x))} \right) - \frac{C}{h_{H}(g^{m+1}(x))}.
		\end{align*}
		Since the $g$-orbit of $x$ is infinite, by the Northcott property, $h_{H}(g^{m}(x))$ goes to $\infty$ as $m \to \infty$. So there is $m_{0}\geq 0$ such that for all $m \in \Sigma \cap [m_{0}, +\infty)$ we have
		\begin{align}\label{eq_est1}
			\frac{h_{H}(g^{m+2}(x))}{h_{H}(g^{m+1}(x))}  \geq  \frac{\varepsilon^{n_{1}} \delta_{N}^{n_{1}}}{ \delta_{N-1}^{n_{1}} } - \frac{2}{\delta_{N-1}^{n_{1}}}. 
		\end{align} 
		By the assumption $(\ast)$, the set
		\begin{align*}
			\left\{ m \geq 0 \ \middle|\ g^{m}(x) \in W \right\}
		\end{align*}
		is finite.
		Since $h_{H}(g^{m}(x))$ goes to $\infty$ as $m \to \infty$, by (\ref{ltd:ineq3}) and by making $m_{0}$ larger if necessary, we may assume 
		\begin{align}\label{eq_est2}
			h_{H}(g^{m+2}(x)) \geq \frac{\kappa}{2} h_{H}(g^{m+1}(x)) 
		\end{align}
		for all $m \geq m_{0}$.
		
		Combining (\ref{eq_est1}) with (\ref{eq_est2}), for $m \geq m_{0} + 2$, we get
		\begin{equation}\label{eq_est_main} 
			\begin{aligned}
				&h_{H}(g^{m}(x))^{1/m} \\
				\geq \ &\left(  \left(\frac{\varepsilon \delta_{N}}{ \delta_{N-1}}\right)^{n_{1}} - \frac{2}{ \delta_{N-1}^{n_{1}}}  \right)^{ \frac{\# (\Sigma \cap [m_{0},m-2])}{m}}
				\left(\frac{\kappa}{2} \right)^{ \frac{\# ([m_{0},m-2] \setminus \Sigma)}{m}} h_{H}(g^{m_{0}+1}(x))^{1/m}.
			\end{aligned} 
		\end{equation}
		Since $\Sigma$ has natural density $1$, we have
		\begin{align*}
			&\lim_{m\to \infty} \frac{\# (\Sigma \cap [m_{0},m-2])}{m} = 1 \\
			&\lim_{m\to \infty} \frac{\# ([m_{0},m-2] \setminus \Sigma)}{m} = 0. 
		\end{align*}
		Thus taking $m \to \infty$ in (\ref{eq_est_main}), we get
		\begin{align*}
			\underline{\alpha}_{g}(x) = \liminf_{m \to \infty}h_{H}(g^{m}(x))^{1/m}  \geq \left(\frac{\varepsilon \delta_{N}}{ \delta_{N-1}}\right)^{n_{1}} - \frac{2}{ \delta_{N-1}^{n_{1}}}.
		\end{align*}
		Note that $\underline{\alpha}_{g}(x) = \underline{\alpha}_{f}(x)^{n_{1}}$ as $g = f^{n_1}$. It then follows that 
		\begin{align*}
			\underline{\alpha}_{f}(x) \geq  \left( \left(\frac{\varepsilon \delta_{N}}{ \delta_{N-1}}\right)^{n_{1}} - \frac{2}{ \delta_{N-1}^{n_{1}}} \right)^{1/n_{1}}.
		\end{align*}
		Since $\varepsilon\delta_{N} > 1$ by the choice of $\varepsilon$, 
		\begin{align*}
			\underline{\alpha}_{f}(x) \geq \lim_{n_{1} \to \infty} \left( \left(\frac{\varepsilon \delta_{N}}{ \delta_{N-1}}\right)^{n_{1}} - \frac{2}{ \delta_{N-1}^{n_{1}}} \right)^{1/n_{1}} = \frac{\varepsilon \delta_{N}}{ \delta_{N-1}}.   
		\end{align*} 
		Since $\varepsilon$ can be arbitrarily close to $1$, we get
		\[ \underline{\alpha}_{f}(x) \geq \frac{\delta_{N}}{ \delta_{N-1}}. \qedhere \]
	\end{proof}
	
	Finally we prove the following boundedness result about preperiodic points.
	
	\begin{theorem}[Theorem \ref{thm:fin-preper-largetopdeg}]
		Let $X$ be a smooth projective variety of dimension $N$ defined over $\qb$.
		Let $f \colon X \dashrightarrow X$ be an $N$-cohomologically hyperbolic dominant ratioal map.
		Then there is a non-empty Zariski open subset $U \subset X$ such that
		the set
		\begin{align*}
			\left\{ x \in U(\qb) \cap X_f(\qb) \,\middle|\, \OO_f(x) \subset U, \# \OO_f(x) < \infty \right\}
		\end{align*}
		is a set of bounded height.
	\end{theorem}
	\begin{proof}
		Fix an ample Cartier divisor $H$ on $X$ and fix an associated height function
		$h_H$ so that $h_H \geq 1$. Here we note that the assertion is independent of the choice of $H$ or $h_H$. 
		
		Write $\delta_i = \delta_i(f)$. Take $\epsilon \in (0,1)$ such that 
		\begin{align*}
			\frac{\epsilon \delta_{N}}{\delta_{N-1}} > 1.
		\end{align*}
		By Proposition \ref{prop_hinq} and Lemma \ref{lem:bigness-largetopdeg},
		as in the proof of Theorem \ref{thm:ad-largetopdeg},
		we can find a positive integer $n_0 \geq 1$, a non-empty Zariski open subset
		$U \subset X$, and a real constant $C \geq 0$ with the following properties:
		\begin{align}
			\left( \frac{\epsilon \delta_{N}}{\delta_{N-1}}\right)^{n_0} >
			1 + \frac{1}{\delta_{N-1}^{n_0}},\label{ineq:keyineq}
		\end{align}
		and
		\begin{align}
			& h_H(f^{2n_0}(y)) + \frac{h_H(y)}{\delta_{N-1}^{n_0}} 
			\geq  \left( \frac{\epsilon \delta_{N}}{\delta_{N-1}}\right)^{n_0} 
			h_H(f^{n_0}(y)) - C, \label{ineq:htineqNhyp} \\
			&h_H(f^{n_0}(y)) \geq h_H(y) - C, \label{ineq:htineqeasyNhyp}
		\end{align}
		for all $y \in U(\qb) \cap X_{f^{n_0}}(\qb)$.
		
		Now take any point $x \in U(\qb) \cap X_f(\qb)$ such that 
		$\OO_f(x) \subset U$ and $\# \OO_f(x) < \infty$.
		Then we obviously have $\OO_{f^{n_0}}(x) \subset U(\qb) \cap X_{f^{n_0}}(\qb)$.
		Let $\alpha$ be the largest root of the quadratic equation
		\begin{align*}
			t^2 - \left( \frac{\epsilon \delta_{N}}{\delta_{N-1}}\right)^{n_0} t + \frac{1}{\delta_{N-1}^{n_0}}=0.
		\end{align*}
		By (\ref{ineq:keyineq}), this equation indeed has two real roots and $\alpha > 1$. We set
		\begin{align*}
			l_n = h_H((f^{n_0})^n(x)) 
			- \frac{1}{\alpha \delta_{N-1}^{n_0}} h_H((f^{n_0})^{n-1}(x))
			- \frac{C}{\alpha -1}
		\end{align*}
		for $n \geq 1$.
		Then by (\ref{ineq:htineqNhyp}), we can show with a little calculation that
		\begin{align*}
			l_{n+1} \geq \alpha l_n
		\end{align*}
		for all $n \geq 1$.
		Hence $l_{n+1} \geq \alpha^n l_1$.
		As $\OO_{f^{n_0}}(x)$ is finite, we have
		\begin{align*}
			\lim_{n \to \infty} \frac{l_{n+1}}{\alpha^n} = 0.
		\end{align*}
		It then follows that $l_1 \leq 0$. Combining with (\ref{ineq:htineqeasyNhyp}), we obtain 
		\begin{align*}
			0 &\geq h_H(f^{n_0}(x)) - \frac{1}{\alpha \delta_{N-1}^{n_0}}h_H(x)
			- \frac{C}{\alpha -1}\\
			& \geq \left(1 - \frac{1}{\alpha \delta_{N-1}^{n_0}} \right) h_H(x)
			-C - \frac{C}{\alpha -1}, 
		\end{align*}
		or equivalently, 
		\begin{align*} 
			h_H(x) \leq \frac{\alpha C}{\alpha+1}\cdot \left(1 - \frac{1}{\alpha \delta_{N-1}^{n_0}} \right)^{-1} 
		\end{align*}
		as $\alpha \delta_{N-1}^{n_0} > 1$. Since all the occurring constants are independent of the point $x$, we are done. 
	\end{proof}

	\bibliographystyle{acm}
	\bibliography{main}

\end{document}